\documentclass{amsart}
\usepackage{amsmath,amssymb}
\usepackage{mathrsfs}
\newtheorem{df}{Definition}

\newtheorem{thm}{Theorem}
\newtheorem{lem}{Lemma}
\newtheorem{rem}{Remark}
\newtheorem{ex}{Example}
\newcommand{\relmiddle}[1]{\mathrel{}\middle#1\mathrel{}}
\numberwithin{equation}{section}
\numberwithin{figure}{section}
\numberwithin{prp}{section}
\numberwithin{df}{section}
\numberwithin{thm}{section}
\numberwithin{lem}{section}
\numberwithin{rem}{section}
\numberwithin{ex}{section}
\begin{document}

\title[primed tableaux and signed unimodal factorizations]
{$\mathfrak{q}$-crystal structure on primed tableaux and on signed unimodal factorizations of reduced words of type $B$}
\author{Toya Hiroshima}
\address{Department of Pure and Applied Mathematics,
Graduate School of Information Science and Technology,
Osaka University,
1-5 Yamadaoka, Suita, Osaka 565-0871, Japan}
\email{t-hiroshima@ist.osaka-u.ac.jp}
\date{}

\begin{abstract}
Crystal basis theory for the queer Lie superalgebra was developed in \cite{GJKKK2,GJKKK3}, where it was shown that semistandard decomposition tableaux admit the structure of crystals for the queer Lie superalgebra or simply $\mathfrak{q}$-crystal structure.
In this paper, we explore the $\mathfrak{q}$-crystal structure of primed tableaux~\cite{HPS} (semistandard marked shifted tableaux~\cite{Cho}) and that of signed unimodal factorizations of reduced words of type $B$~\cite{HPS}.
We give the explicit odd Kashiwara operators on primed tableaux and the forms of the highest and lowest weight vectors.
We clarify the relation between signed unimodal factorizations and the type $B$ Coxeter-Knuth relation of reduced words.
We also give the explicit algorithms for odd Kashiwara operators on signed unimodal factorizations of reduced words of type $B$.
\end{abstract}


\maketitle

\section{Introduction}
Several combinatorial models have been proposed and investigated in the study of Schur $P$- or $Q$-functions~\cite{Cho,CK,GJKKK2,GJKKK3,Sag,Ser,Stem}.
Among them, semistandard decomposition tableaux were shown to admit the structure of crystals for the queer Lie superalgebra or simply $\mathfrak{q}$-crystal structure by Grantcharov et al.~\cite{GJKKK2,GJKKK3}.
On the other hand, Hawkes et al.~\cite{HPS} gave a bijection between the set of semistandard unimodal tableaux and that of primed tableaux of the same shape.
The semistandard unimodal tableaux are the same as the semistandard decomposition tableaux originally introduced by Serrano~\cite{Ser} and a bijection between the set of semistandard decomposition tableaux defined in \cite{GJKKK2,GJKKK3} and that defined in \cite{Ser} was established by Choi et al.~\cite{CNO}.
Combining these results, it follows that primed tableaux form a $\mathfrak{q}$-crystal.
Recently, Assaf and O\u{g}uz~\cite{AO1,AO2} constructed the even (or ordinary) Kashiwara operators on semistandard marked shifted tableaux by using Stembridge's local axioms.
They further augmented the even Kashiwara operators they gave to the odd Kashiwara operators and constructed a $\mathfrak{q}$-crystal structure on semistandard marked shifted tableaux~\cite{AO1,AO2}.
We construct the odd Kashiwara operators on primed tableaux along the lines of \cite{HPS}.
That is, we give the odd Kashiwara operators on primed tableaux by translating the $\mathfrak{q}$-crystal structure of words through Serrano's semistandard mixed shifted insertion~\cite{Ser}.
The method of \cite{AO1,AO2} and the one in this paper are completely different but the results are the same. 
By exploiting the $\mathfrak{q}$-crystal structure of primed tableaux thus obtained, we  impose a $\mathfrak{q}$-crystal structure on signed unimodal factorizations~\cite{HPS} (with the fixed number of factors) of reduced words of type $B$.

In more detail, the main results of the present paper are summarized as follows.
Firstly, as in \cite{HPS}, we employ the semistandard mixed shifted insertion to show that the set of primed tableaux forms a $\mathfrak{q}$-crystal and that the odd Kashiwara operators on a primed tableau are independent of a choice of a recording tableau in the above insertion scheme (Theorem~\ref{thm:primed}). 
%The explicit rules of the action of odd Kashiwara operators are given in Lemma~\ref{lem:eP} and \ref{lem:fP}.
Secondly, we calculate the highest and lowest weight vectors of primed tableaux by using the rules of ordinary Kashiwara operators on a primed tableau developed in \cite{HPS} (Theorem~\ref{thm:highlow}).
Thirdly, by the $\mathfrak{q}$-crystal structure of primed tableaux thus obtained and by the primed Kra\'{s}kiewicz insertion introduced in \cite{HPS}, 
we show that the set of signed unimodal factorizations (with the fixed number of factors) of type $B$ Coxeter-Knuth related reduced words admits a $\mathfrak{q}$-crystal structure.
Namely, the primed Kra\'{s}kiewicz insertion gives a $\mathfrak{q}$-crystal isomorphism (Theorem~\ref{thm:factorization}).
We also give the explicit rules of odd Kashiwara operators on a signed unimodal factorization of reduced words of type $B$ 
(Theorem~\ref{thm:eF} and \ref{thm:fF}).

The paper is organized as follows.
In Section~\ref{sec:queer}, we review $\mathfrak{q}$-crystals and the model of semistandard decomposition tableaux as a realization of the $\mathfrak{q}$-crystal.
The main combinatorial objects in Section~\ref{sec:primed} are primed tableaux or semistandard marked shifted tableaux.
We give the odd Kashiwara operators on a primed tableau by using the semistandard mixed shifted insertion  
and also give the forms of the highest and lowest weight vectors of primed tableaux.
Section~\ref{sec:unimodal} concerns signed unimodal factorizations of reduced words of type $B$.
We point out that such factorizations with the fixed number of factors form a $\mathfrak{q}$-crystal and give the explicit algorithms for odd Kashiwara operators.

\section{Crystals for the queer Lie superalgebra} \label{sec:queer}

Let us start by briefly sketching crystals for the general linear Lie algebra $\mathfrak{gl}(n)$ or simply $\mathfrak{gl}(n)$-crystals~\cite{BS,HK,Kas2}.
Let 
$P=
\bigoplus _{i=1}^{n}
\mathbb{Z\epsilon}_{i}$ 
be the weight lattice and 
$P^{\vee}=
\bigoplus _{i=1}^{n}
\mathbb{Z}k_{i}$ 
be the dual weight lattice with 
$\left\langle \epsilon_{i},k_{j}\right\rangle =\delta_{ij}$ for ($i,j=1,\ldots,n$).
Let 
$\left\{  \alpha_{i}=\epsilon_{i}-\epsilon_{i+1} \mid 1\leq i <n \right\}  $ be the set of simple roots and 
$\left\{  h_{i}=k_{i}-k_{i+1} \mid 1\leq i <n \right\}  $ be the set of simple coroots.
Let
\[
P^{+}=\left\{  \lambda \mid \lambda\in P,\left\langle \lambda,h_{i}\right\rangle \geq0 \; (i=1,\ldots,n-1) \right\}
\] 
be the set of dominant integral weights.

\begin{df} \label{df:crystal}
A set $B$ together with maps
$\mathrm{wt}:B\rightarrow P$ and 
$\Tilde{e_{i}},\Tilde{f_{i}}:B\rightarrow B\sqcup\{\boldsymbol{0}\}$
is called a $\mathfrak{gl}(n)$-crystal if the following properties are satisfied for $i=1,\ldots,n-1$:
when we define 
\[
\varepsilon_{i}(b)=\max\left\{  k\geq0 \relmiddle| \Tilde{e_{i}}^{k}b\in B\right\},
\]
and
\[
\varphi_{i}(b)=\max\left\{  k\geq0 \relmiddle| \Tilde{f_{i}}^{k}b\in B\right\},
\]
for $b\in B$, then
\begin{itemize}
\item[(1)]
$\varepsilon_{i},\varphi_{i}:B\rightarrow\mathbb{Z}_{\geq0}$ 
and 
$\varphi_{i}(b)=\varepsilon_{i}(b)+
\left\langle \mathrm{wt}(b), h_{i} \right\rangle $,
\item[(2)]
if $\Tilde{e_{i}}b\neq \boldsymbol{0}$, then 
$\mathrm{wt}(\Tilde{e_{i}}b)=\mathrm{wt}(b)+\alpha_{i}$, 
$\varepsilon_{i}(\Tilde{e_{i}}b)=\varepsilon_{i}(b)-1$, and 
$\varphi_{i}(\Tilde{e_{i}}b)=\varphi_{i}(b)+1$,
\item[(3)]
if $\Tilde{f_{i}}b\neq \boldsymbol{0}$, then 
$\mathrm{wt}(\Tilde{f_{i}}b)=\mathrm{wt}(b)-\alpha_{i}$, 
$\varepsilon_{i}(\Tilde{f_{i}}b)=\varepsilon_{i}(b)+1$, and 
$\varphi_{i}(\Tilde{f_{i}}b)=\varphi_{i}(b)-1$,
\item[(4)]
for $b,b^{\prime}\in B$, 
$\Tilde{f_{i}}b=b^{\prime}\Longleftrightarrow\Tilde{e_{i}}b^{\prime}=b$.
\end{itemize}
\end{df}
The maps $\Tilde{e_{i}}$ and $\Tilde{f_{i}}$ are called Kashiwara operators 
and $\mathrm{wt}(b)$ is called the weight of $b$.
In Definition~\ref{df:crystal}, $\boldsymbol{0}$ is a formal symbol; $b=\boldsymbol{0}$ implies $b\notin B$.

\begin{rem}
More precisely, crystals in Definition~\ref{df:crystal} are called semiregular~\cite{HK} or seminormal~\cite{BS}.
Throughout this paper, we consider only semiregular or seminormal crystals.
\end{rem}

\begin{df}[tensor product rule] \label{df:tensor}
Let $B_{1}$ and $B_{2}$ be $\mathfrak{gl}(n)$-crystals.
The tensor product $B_{1}\otimes B_{2}$ is defined to be 
the set 
$B_{1}\times B_{2} = 
\left\{  b_{1}\otimes b_{2} \mid b_{1}\in B_{1},b_{2}\in B_{2}\right\}  $
whose crystal structure is defined by
\begin{itemize}
\item[(1)]
$\mathrm{wt}(b_{1}\otimes b_{2})=\mathrm{wt}(b_{1})+\mathrm{wt}(b_{2})$,
\item[(2)]
$\varepsilon_{i}(b_{1}\otimes b_{2})=
\max\left\{  \varepsilon_{i}(b_{2}),\varepsilon _{i}(b_{1})-\left\langle \mathrm{wt}(b_{2}),h_{i} \right\rangle \right\}$,
\item[(3)]
$\varphi_{i}(b_{1}\otimes b_{2})=
\max\left\{ \varphi_{i}(b_{1}), \varphi_{i}(b_{2})+\left\langle \mathrm{wt}(b_{1}), h_{i} \right\rangle \right\}$,
\item[(4)]
$\Tilde{e_{i}}(b_{1}\otimes b_{2})=
\begin{cases}
\Tilde{e_{i}}b_{1}\otimes b_{2} &  if \varphi_{i}(b_{2})<\varepsilon_{i}(b_{1}), \\
b_{1}\otimes\Tilde{e_{i}}b_{2} &  if \varphi_{i}(b_{2})\geq\varepsilon_{i}(b_{1}),
\end{cases}
$
\item[(5)]
$\Tilde{f_{i}}(b_{1}\otimes b_{2})=
\begin{cases}
\Tilde{f_{i}}b_{1}\otimes b_{2} & if \varphi_{i}(b_{2})\leq\varepsilon_{i}(b_{1}), \\
b_{1}\otimes\Tilde{f_{i}}b_{2} & if \varphi_{i}(b_{2})>\varepsilon_{i}(b_{1}),
\end{cases}
$
\end{itemize}
for $i=1,\ldots,n-1$.
\end{df}

\begin{rem}
Note that this definition is different from the one by the convention of Kashiwara.
Since we use Kashiwara operators on primed tableaux which are constructed by the anti-Kashiwara convention for the tensor product rule~\cite{HPS}, we adopt Definition~\ref{df:tensor}.
\end{rem}

Next, let us describe crystals for the queer Lie superalgebra $\mathfrak{q}(n)$ or simply $\mathfrak{q}(n)$-crystals introduced in \cite{GJKKK2,GJKKK3}.

\begin{df} \label{df:queer}
A $\mathfrak{q}(n)$-crystal is a set $B$ together with the maps $\mathrm{wt}:B\rightarrow P$, 
$\varepsilon_{i}, \varphi_{i}:B\rightarrow\mathbb{Z}_{\geq 0}$ and 
$\Tilde{e}_{i},\Tilde{f}_{i}:B\rightarrow B\sqcup\{\boldsymbol{0}\}$ for 
$i\in I:=\{1,\ldots,n-1,\Bar{1}\}$ 
satisfying the following conditions:

\begin{itemize}
\item[(1)]
$B$ is a $\mathfrak{gl}(n)$-crystal with respect to $\mathrm{wt}$, $\varepsilon_{i}$, $\varphi_{i}$, $\Tilde{e}_{i}$ 
and $\Tilde{f}_{i}$ for $i=1,\ldots,n-1$,
\item[(2)]
$\mathrm{wt}(b)\in
\bigoplus _{i=1}^{n}
\mathbb{Z}_{\geq 0}\epsilon_{i}$ 
for $b\in B,$
\item[(3)] $\mathrm{wt}(\Tilde{e}_{\Bar{1}}b)=\mathrm{wt}(b)+\alpha_{1}$, 
$\mathrm{wt}(\Tilde{f}_{\Bar{1}}b)=\mathrm{wt}(b)-\alpha_{1}$ 
for $b\in B$,
\item[(4)]
$\Tilde{f}_{\Bar{1}}b=b^{\prime}$ if and only if 
$b=\Tilde{e}_{\Bar{1}}b^{\prime}$ for all $b,b^{\prime}\in B$,
\item[(5)] for $3\leq i\leq n-1$, we have 
\begin{itemize}
\item[(i)] the operators $\Tilde{e}_{\Bar{1}}$ and $\Tilde{f}_{\Bar{1}}$ commute with $\Tilde{e}_{i}$ and 
$\Tilde{f}_{i}$,
\item[(ii)] if $\Tilde{e}_{\Bar{1}}b\in B$, then 
$\varepsilon_{i}(\Tilde{e}_{\Bar{1}}b)=\varepsilon_{i}(b)$ and 
$\varphi_{i}(\Tilde{e}_{\Bar{1}}b)=\varphi_{i}(b)$.
\end{itemize}
\end{itemize}
\end{df}

The crystal associated with an irreducible highest weight module $V(\lambda)$ in the category $\mathcal{O}_{int}^{\geq 0}$~\cite{GJKKK1,GJKK} of tensor representations over the $\mathfrak{q}(n)$-quantum group $U_{q}(\mathfrak{q}(n))$ is denoted by $B_{n}(\lambda)$.
We identify the partition $\lambda$ with $\lambda=\sum_{i=1}^{n}\lambda_{i}\epsilon_{i}\in P^{+}$ in a usual way~\cite{HK}.
In $B_{n}(\lambda)$, the relevant partition $\lambda$ is a strict partition $\lambda=(\lambda_{1} > \lambda_{2} > \ldots > \lambda_{l} > \lambda_{l+1}= 0)$, where $l\leq n$.
A crystal $B$ can be viewed as an oriented colored graph with colors $i\in I$ 
when we define $b\stackrel{i}{\longrightarrow} b^{\prime}$ if 
$\Tilde{f_{i}}b=b^{\prime}\; (b,b^{\prime}\in B)$.
This graph is called a crystal graph.
The crystal graph of $B_{n}(\square)$, i.e., $B_{n}(\lambda)$ for $\lambda=\epsilon_{1}\in P^{+}$, is given by
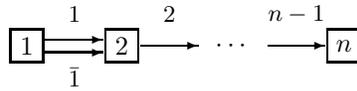
\setlength{\unitlength}{12pt}

\begin{center}
\begin{picture}(11,3)

\put(1.1,1.3){\vector(1,0){1.8}}
\put(1.1,1.7){\vector(1,0){1.8}}
\put(4.1,1.5){\vector(1,0){1.8}}
\put(8.1,1.5){\vector(1,0){1.8}}

\put(6,1){\makebox(2,1){$\cdots$}}

\put(0,1){\framebox(1,1){$1$}}
\put(3,1){\framebox(1,1){$2$}}
\put(10,1){\framebox(1,1){$n$}}

\put(1,0){\makebox(2,1){\small $\Bar{1}$}}
\put(1,2){\makebox(2,1){\small $1$}}
\put(4,2){\makebox(2,1){\small $2$}}
\put(8,2){\makebox(2,1){\small $n-1$}}

\end{picture}.
\end{center}

For $\mathfrak{q}(n)$-crystals $B_{1}$ and $B_{2}$, the tensor product $B_{1}\otimes B_{2}$ is a $\mathfrak{gl}(n)$-crystal, where actions of $\Tilde{e}_{\Bar{1}}$ and $\Tilde{f}_{\Bar{1}}$ on $b_{1}\otimes b_{2}$ ($b_{1}\in B_{1}$, $b_{2}\in B_{2}$) are prescribed as
\[
\Tilde{e}_{\Bar{1}}(b_{1}\otimes b_{2})=
\begin{cases}
b_{1} \otimes \Tilde{e}_{\Bar{1}}b_{2}, & 
if \left\langle \mathrm{wt}(b_{1}), k_{1}\right\rangle =\left\langle \mathrm{wt}(b_{1}), k_{2}\right\rangle=0, \\
\Tilde{e}_{\Bar{1}}b_{1} \otimes b_{2}, & otherwise,
\end{cases}
\]
\[
\Tilde{f}_{\Bar{1}}(b_{1}\otimes b_{2})=
\begin{cases}
b_{1} \otimes \Tilde{f}_{\Bar{1}}b_{2}, & 
if \left\langle \mathrm{wt}(b_{1}), k_{1}\right\rangle =\left\langle \mathrm{wt}(b_{1}), k_{2}\right\rangle=0, \\
\Tilde{f}_{\Bar{1}}b_{1} \otimes b_{2}, & otherwise.
\end{cases}
\]
Under these rules,  $B_{1}\otimes B_{2}$ is also a $\mathfrak{q}(n)$-crystal.

\begin{rem}
These rules are different from the ones given in \cite{CK,GJKKK2,GJKKK3} because we adopt the anti-Kashiwara convention for the tensor product rule.
\end{rem}

Let $B$ be a $\mathfrak{q}(n)$-crystal and suppose that $B$ is in a class of normal $\mathfrak{gl}(n)$-crystal~\cite{BS}, i.e., every connected component in $B$ has a unique highest weight element.
Here, the connected components in $B$ are referred to as the maximal subcrystals of $B$ where all the elements are connected by even Kashiwara operators.
We define the automorphism $S_{i}$ on $B$ by
\[
S_{i}(b)=
\begin{cases}
\Tilde{f}_{i}^{\left\langle \mathrm{wt}(b),h_{i}\right\rangle }b, & 
if \left\langle \mathrm{wt}(b),h_{i}\right\rangle \geq0, \\
\Tilde{e}_{i}^{-\left\langle \mathrm{wt}(b),h_{i}\right\rangle }b, & 
if \left\langle \mathrm{wt}(b),h_{i}\right\rangle <0,
\end{cases}
\]
for $b\in B$ and $i=1,2,\ldots,n-1$.
Let $w$ be an element of Weyl group $W$ of $\mathfrak{gl}(n)$ which is generated by simple reflections $s_{i}$ for $i=1,\ldots,n-1$.
Then, there exists a unique action $S_{w}:B\rightarrow B$ of $W$ on $B$ such that $S_{s_{i}}=S_{i}$ 
for $i=1.\ldots,n-1$~\cite{Kas2}.
Let us set $w_{i}=s_{2}\cdots s_{i}s_{1}\cdots s_{i-1}$.
Then $w_{i}$ is the shortest element in $W$ such that $w_{i}(\alpha_{i})=\alpha_{1}$ 
for $i=2,\ldots,n-1$.
We define new operators by
\[
\Tilde{e}_{\Bar{\imath}} =S_{w_{i}^{-1}}\Tilde{e}_{\Bar{1}}S_{w_{i}},\quad  
\Tilde{f}_{\Bar{\imath}} =S_{w_{i}^{-1}}\Tilde{f}_{\Bar{1}}S_{w_{i}},
\]
where $S_{w_{i}}=S_{2}\cdots S_{i}S_{1}\cdots S_{i-1}$ and similar for $S_{w_{i}^{-1}}$.
These operators together with $\Tilde{e}_{\Bar{1}}$ and $\Tilde{f}_{\Bar{1}}$ are called \emph{odd} Kashiwara operators, while $\Tilde{e}_{i}$ and $\Tilde{f}_{i}$ ($i=1,\ldots,n-1$) are called \emph{even} Kashiwara operators.

\begin{thm}[\cite{GJKKK2,GJKKK3}] 
Let $B_{n}(\lambda)$ be a $\mathfrak{q}(n)$-crystal.
There is a unique element $b\in B_{n}(\lambda)$ such that
$\Tilde{e}_{i}b=\Tilde{e}_{\Bar{\imath}}b=\boldsymbol{0}$ for all $i=1,\ldots,n-1$, 
which is called a \emph{$\mathfrak{q}(n)$-highest weight vector} and 
there is a unique element $b\in B_{n}(\lambda)$ called a \emph{$\mathfrak{q}(n)$-lowest weight vector} such that 
$S_{w_{0}}b$ is a $\mathfrak{q}(n)$-highest weight vector, where $w_{0}$ is the longest element of $W$.
\end{thm}

One of the combinatorial models of $\mathfrak{q}(n)$-crystals is the model of semistandard decomposition tableaux~\cite{GJKKK2,GJKKK3}, which we will describe in the rest of this section.
To begin with, we introduce necessary notation and definitions, some of which will be used in subsequent sections.

Let $\mathcal{P}^{+}$ denote the set of strict partitions, 
$\lambda=(\lambda_{1}>\lambda_{2}>\cdots>\lambda_{l}>\lambda_{l+1}=0)$.
For $\lambda\in\mathcal{P}^{+}$, the length $l(\lambda)$ of $\lambda$ is defined as the number of positive parts of $\lambda$ 
and the size of $\lambda$ is defined to be $\left\vert \lambda\right\vert =\sum_{i=1}^{l(\lambda)}\lambda_{i}$.
For $m=\left\vert \lambda \right\vert$, we write $\lambda\vdash m$.
The \emph{shifted diagram} of shape $\lambda\in\mathcal{P}^{+}$ is defined to be the set 
\[
S(\lambda)=\{(i,j)\in\mathbb{N}^{2} \mid i\leq j\leq\lambda_{i}+i-1,1\leq i\leq
l(\lambda)\}.
\]
A filling $T$ of $S(\lambda)$ with letters is called a \emph{shifted tableau} where the entry at $(i,j)$-position is denoted by $T_{i,j}$.

\begin{df} \label{df:SSDT}
\begin{itemize}
\item[(1)]
A word $\boldsymbol{a}=a_{1}\cdots a_{s}$ 
is called a \emph{hook word} if it satisfies 
$a_{1}\geq a_{2}\geq\cdots\geq a_{k}<a_{k+1}<\cdots<a_{s}$ for some $1\leq k\leq s$.
%In this case, let 
%$\boldsymbol{a}_{\downarrow}=a_{1}\cdots a_{k}$ be the weakly decreasing subword of maximal length and 
%$\boldsymbol{a}_{\uparrow}=a_{k+1}\cdots a_{s}$ the remaining strictly increasing subword in $\boldsymbol{a}$.
\item[(2)]
For $\lambda\in \mathcal{P}^{+}$, let $T$ be a filling of $S(\lambda)$ with letters from the alphabet $X_{n}:=\{1,2,\ldots,n\}$.
Then, $T$ is called a \emph{semistandard decomposition tableau} of shape $\lambda$ if 
\begin{itemize}
\item[(i)]
$v^{(i)}=T_{i,i}T_{i,i+1}\cdots T_{i,i+\lambda_{i}-1}$ is a hook word of length $\lambda_{i}$ for $i=1,\ldots, l(\lambda)$.
\item[(ii)]
$v^{(i)}$ is a hook subword of maximal length in $v^{(i+1)}v^{(i)}$, the concatenation of $v^{(i+1)}$ and $v^{(i)}$, for 
$i=1,\ldots,l(\lambda)-1$.
\end{itemize}
\end{itemize}
We denote by $\mathrm{SSDT}_{n}(\lambda)$  the set of semistandard decomposition tableaux of shape $\lambda$.
\end{df}

\begin{rem} \label{rem:plethora}
Definition~\ref{df:SSDT} is due to Grantcharov, Jung, Kang, Kashiwara, and Kim~\cite{GJKKK2} and 
is different from the one originally introduced in \cite{Ser} and the one used in \cite{Cho,HPS}.
However, the bijection among these two models is established in \cite{CNO}, where the word and the tableau defined in Definition~\ref{df:SSDT} are called the reverse hook word and the reverse semistandard decomposition tableau, respectively.
The hook word and the semistandard decomposition tableau defined in \cite{Cho,Ser} are also called the weakly unimodal word and the semistandard unimodal tableau, respectively~\cite{HPS}. 
The definition here is more adequate to describe the $\mathfrak{q}(n)$-crystal structure because the $\mathfrak{q}(n)$-highest and lowest weight vectors take simpler forms. 
\end{rem}

We define the \emph{reading word} of $T\in \mathrm{SSDT}_{n}(\lambda)$ as follows:
We read entries through each row from right to left starting from the topmost row to the bottommost one.
We denote by $\mathrm{rw}(T)$ the reading word of $T$.
This definition is opposite to that given in \cite{CK,GJKKK2} because we follow the anti-Kashiwara convention for the tensor product rule.
The weight of $T$ is defined to be $\mathrm{wt}(T)=\sum _{i=1}^{n}a_{i}\epsilon _{i}$, 
where $a_{i}$ is the number of letters $i$ in $\mathrm{rw}(T)$.

Let $\mathrm{rw}(T)=a_{1}a_{2}\cdots a_{m}$ be the reading word of $T\in \mathrm{SSDT}_{n}(\lambda)$.
We identify $T$ with the tensor product,
\begin{center}
\begin{picture}(6,1)

\put(0,0){\framebox(1,1){$a_{1}$}}
\put(5,0){\framebox(1,1){$a_{m}$}}

\put(1,0){\makebox(1,1){$\otimes$}}
\put(4,0){\makebox(1,1){$\otimes$}}
\put(2,0){\makebox(2,1){$\cdots$}}

\end{picture}
\end{center}
by embedding 
\[
\mathrm{SSDT}_{n}(\lambda)\rightarrow B_{n}(\square)^{\otimes \left\vert \lambda \right\vert}.
\]
The $\mathfrak{q}(n)$-crystal structure of $B_{n}(\square)^{\otimes \left\vert \lambda \right\vert}$ is given by the crystal graph of $B_{n}(\square)$ equipped with the tensor product rule.
Thus, $\mathrm{SSDT}_{n}(\lambda)$ is endowed with the $\mathfrak{q}(n)$-crystal structure (Fig.~\ref{fig:SSDT}).
The actual computation of the crystal graph can be done by the bracketing rule on $\mathrm{rw}(T)$ (see Section~\ref{sec:primed}).

\setlength{\unitlength}{10pt}

\begin{figure}

\begin{center}
\begin{picture}(33,26)

\put(1.5,11.8){\vector(0,-1){1.6}}
\put(1.5,15.8){\vector(0,-1){1.6}}

\put(7.5,11.8){\vector(0,-1){1.6}}
\put(7.5,15.8){\vector(0,-1){1.6}}
\put(7.5,19.8){\vector(0,-1){1.6}}

\put(13.5,3.8){\vector(0,-1){1.6}}
\put(13.5,7.8){\vector(0,-1){1.6}}
\put(13.3,11.8){\vector(0,-1){1.6}}
\put(13.7,11.8){\vector(0,-1){1.6}}
\put(13.5,15.8){\vector(0,-1){1.6}}
\put(13.5,19.8){\vector(0,-1){1.6}}
\put(13.5,23.8){\vector(0,-1){1.6}}

\put(19.3,7.8){\vector(0,-1){1.6}}
\put(19.7,7.8){\vector(0,-1){1.6}}
\put(19.5,11.8){\vector(0,-1){1.6}}
\put(19.5,15.8){\vector(0,-1){1.6}}
\put(19.5,19.8){\vector(0,-1){1.6}}

\put(31.5,11.8){\vector(0,-1){1.6}}
\put(31.3,15.8){\vector(0,-1){1.6}}
\put(31.7,15.8){\vector(0,-1){1.6}}

\put(3,8){\vector(3,-2){2.5}}
\put(6,12){\vector(-3,-2){2,5}}
\put(6,15.8){\vector(-3,-2){2,5}}
\put(6,16.2){\vector(-3,-2){2,5}}
\put(6,19.8){\vector(-3,-2){2,5}}
\put(6,20.2){\vector(-3,-2){2,5}}

\put(9,3.8){\vector(3,-2){2.5}}
\put(9,4.2){\vector(3,-2){2.5}}
\put(9,7.8){\vector(3,-2){2.5}}
\put(9,8.2){\vector(3,-2){2.5}}
\put(12,12){\vector(-3,-2){2,5}}
\put(12,24){\vector(-3,-2){2,5}}

\put(15,16){\vector(3,-2){2.5}}
\put(15,20){\vector(3,-2){2.5}}
\put(15,24){\vector(3,-2){2.5}}

\put(21,16){\vector(3,-2){2.5}}
\put(24,12){\vector(-3,-2){2,5}}

\put(27,12){\vector(3,-2){2.5}}

\put(0.5,10.5){\makebox(1,1){\scriptsize $2$}}
\put(0.5,14.5){\makebox(1,1){\scriptsize $2$}}

\put(6.5,10.5){\makebox(1,1){\scriptsize $1$}}
\put(7.5,14.5){\makebox(1,1){\scriptsize $2$}}
\put(7.5,18.5){\makebox(1,1){\scriptsize $2$}}

\put(13.5,2.5){\makebox(1,1){\scriptsize $2$}}
\put(13.5,6.5){\makebox(1,1){\scriptsize $2$}}
\put(12.3,10.5){\makebox(1,1){\scriptsize $1$}}
\put(13.7,10.5){\makebox(1,1){\scriptsize $\Bar{1}$}}
\put(13.5,14.5){\makebox(1,1){\scriptsize $1$}}
\put(13.5,18.5){\makebox(1,1){\scriptsize $1$}}
\put(13.5,22.5){\makebox(1,1){\scriptsize $2$}}

\put(18.3,6.5){\makebox(1,1){\scriptsize $1$}}
\put(19.7,6.5){\makebox(1,1){\scriptsize $\Bar{1}$}}
\put(18.5,10.5){\makebox(1,1){\scriptsize $2$}}
\put(19.5,14.5){\makebox(1,1){\scriptsize $1$}}
\put(19.5,18.5){\makebox(1,1){\scriptsize $2$}}

\put(31.5,10.5){\makebox(1,1){\scriptsize $2$}}
\put(30.3,14.5){\makebox(1,1){\scriptsize $1$}}
\put(31.7,14.5){\makebox(1,1){\scriptsize $\Bar{1}$}}

\put(3.5,5.8){\makebox(1,1){\scriptsize $2$}}
\put(4.5,9.8){\makebox(1,1){\scriptsize $\Bar{1}$}}
\put(4.5,13.6){\makebox(1,1){\scriptsize $\Bar{1}$}}
\put(4.5,15.7){\makebox(1,1){\scriptsize $1$}}
\put(4.5,17.6){\makebox(1,1){\scriptsize $\Bar{1}$}}
\put(4.5,19.7){\makebox(1,1){\scriptsize $1$}}

\put(9.5,1.8){\makebox(1,1){\scriptsize $\Bar{1}$}}
\put(9.5,3.8){\makebox(1,1){\scriptsize $1$}}
\put(9.5,5.8){\makebox(1,1){\scriptsize $\Bar{1}$}}
\put(9.5,7.8){\makebox(1,1){\scriptsize $1$}}
\put(10,11.5){\makebox(1,1){\scriptsize $2$}}
\put(10,23.5){\makebox(1,1){\scriptsize $1$}}

\put(16,15.5){\makebox(1,1){\scriptsize $\Bar{1}$}}
\put(16,19.5){\makebox(1,1){\scriptsize $\Bar{1}$}}
\put(16,23.5){\makebox(1,1){\scriptsize $\Bar{1}$}}

\put(22.5,9.7){\makebox(1,1){\scriptsize $1$}}
\put(22,15.5){\makebox(1,1){\scriptsize $2$}}

\put(28,11.5){\makebox(1,1){\scriptsize $\Bar{1}$}}

\put(1,8){\makebox(1,1){\small $1$}}
\put(0,9){\makebox(1,1){\small $3$}}
\put(1,9){\makebox(1,1){\small $3$}}
\put(2,9){\makebox(1,1){\small $2$}}

\put(1,12){\makebox(1,1){\small $1$}}
\put(0,13){\makebox(1,1){\small $3$}}
\put(1,13){\makebox(1,1){\small $2$}}
\put(2,13){\makebox(1,1){\small $2$}}

\put(1,16){\makebox(1,1){\small $1$}}
\put(0,17){\makebox(1,1){\small $2$}}
\put(1,17){\makebox(1,1){\small $2$}}
\put(2,17){\makebox(1,1){\small $2$}}

\put(7,4){\makebox(1,1){\small $1$}}
\put(6,5){\makebox(1,1){\small $3$}}
\put(7,5){\makebox(1,1){\small $3$}}
\put(8,5){\makebox(1,1){\small $3$}}

\put(7,8){\makebox(1,1){\small $2$}}
\put(6,9){\makebox(1,1){\small $3$}}
\put(7,9){\makebox(1,1){\small $3$}}
\put(8,9){\makebox(1,1){\small $1$}}

\put(7,12){\makebox(1,1){\small $1$}}
\put(6,13){\makebox(1,1){\small $3$}}
\put(7,13){\makebox(1,1){\small $3$}}
\put(8,13){\makebox(1,1){\small $1$}}

\put(7,16){\makebox(1,1){\small $1$}}
\put(6,17){\makebox(1,1){\small $3$}}
\put(7,17){\makebox(1,1){\small $2$}}
\put(8,17){\makebox(1,1){\small $1$}}

\put(7,20){\makebox(1,1){\small $1$}}
\put(6,21){\makebox(1,1){\small $2$}}
\put(7,21){\makebox(1,1){\small $2$}}
\put(8,21){\makebox(1,1){\small $1$}}

\put(13,0){\makebox(1,1){\small $2$}}
\put(12,1){\makebox(1,1){\small $3$}}
\put(13,1){\makebox(1,1){\small $3$}}
\put(14,1){\makebox(1,1){\small $3$}}

\put(13,4){\makebox(1,1){\small $2$}}
\put(12,5){\makebox(1,1){\small $3$}}
\put(13,5){\makebox(1,1){\small $3$}}
\put(14,5){\makebox(1,1){\small $2$}}

\put(13,8){\makebox(1,1){\small $2$}}
\put(12,9){\makebox(1,1){\small $3$}}
\put(13,9){\makebox(1,1){\small $2$}}
\put(14,9){\makebox(1,1){\small $2$}}

\put(13,12){\makebox(1,1){\small $2$}}
\put(12,13){\makebox(1,1){\small $3$}}
\put(13,13){\makebox(1,1){\small $2$}}
\put(14,13){\makebox(1,1){\small $1$}}

\put(13,16){\makebox(1,1){\small $2$}}
\put(12,17){\makebox(1,1){\small $3$}}
\put(13,17){\makebox(1,1){\small $1$}}
\put(14,17){\makebox(1,1){\small $1$}}

\put(13,20){\makebox(1,1){\small $1$}}
\put(12,21){\makebox(1,1){\small $3$}}
\put(13,21){\makebox(1,1){\small $1$}}
\put(14,21){\makebox(1,1){\small $1$}}

\put(13,24){\makebox(1,1){\small $1$}}
\put(12,25){\makebox(1,1){\small $2$}}
\put(13,25){\makebox(1,1){\small $1$}}
\put(14,25){\makebox(1,1){\small $1$}}

\put(19,4){\makebox(1,1){\small $2$}}
\put(18,5){\makebox(1,1){\small $3$}}
\put(19,5){\makebox(1,1){\small $2$}}
\put(20,5){\makebox(1,1){\small $3$}}

\put(19,8){\makebox(1,1){\small $2$}}
\put(18,9){\makebox(1,1){\small $3$}}
\put(19,9){\makebox(1,1){\small $1$}}
\put(20,9){\makebox(1,1){\small $3$}}

\put(19,12){\makebox(1,1){\small $2$}}
\put(18,13){\makebox(1,1){\small $3$}}
\put(19,13){\makebox(1,1){\small $1$}}
\put(20,13){\makebox(1,1){\small $2$}}

\put(19,16){\makebox(1,1){\small $1$}}
\put(18,17){\makebox(1,1){\small $3$}}
\put(19,17){\makebox(1,1){\small $1$}}
\put(20,17){\makebox(1,1){\small $2$}}

\put(19,20){\makebox(1,1){\small $1$}}
\put(18,21){\makebox(1,1){\small $2$}}
\put(19,21){\makebox(1,1){\small $1$}}
\put(20,21){\makebox(1,1){\small $2$}}

\put(25,12){\makebox(1,1){\small $1$}}
\put(24,13){\makebox(1,1){\small $3$}}
\put(25,13){\makebox(1,1){\small $1$}}
\put(26,13){\makebox(1,1){\small $3$}}

\put(31,8){\makebox(1,1){\small $1$}}
\put(30,9){\makebox(1,1){\small $3$}}
\put(31,9){\makebox(1,1){\small $2$}}
\put(32,9){\makebox(1,1){\small $3$}}

\put(31,12){\makebox(1,1){\small $1$}}
\put(30,13){\makebox(1,1){\small $2$}}
\put(31,13){\makebox(1,1){\small $2$}}
\put(32,13){\makebox(1,1){\small $3$}}

\put(31,16){\makebox(1,1){\small $1$}}
\put(30,17){\makebox(1,1){\small $2$}}
\put(31,17){\makebox(1,1){\small $1$}}
\put(32,17){\makebox(1,1){\small $3$}}

\put(0,9){\line(0,1){1}}
\put(1,8){\line(0,1){2}}
\put(2,8){\line(0,1){2}}
\put(3,9){\line(0,1){1}}
\put(1,8){\line(1,0){1}}
\put(0,9){\line(1,0){3}}
\put(0,10){\line(1,0){3}}

\put(0,13){\line(0,1){1}}
\put(1,12){\line(0,1){2}}
\put(2,12){\line(0,1){2}}
\put(3,13){\line(0,1){1}}
\put(1,12){\line(1,0){1}}
\put(0,13){\line(1,0){3}}
\put(0,14){\line(1,0){3}}

\put(0,17){\line(0,1){1}}
\put(1,16){\line(0,1){2}}
\put(2,16){\line(0,1){2}}
\put(3,17){\line(0,1){1}}
\put(1,16){\line(1,0){1}}
\put(0,17){\line(1,0){3}}
\put(0,18){\line(1,0){3}}

\put(6,5){\line(0,1){1}}
\put(7,4){\line(0,1){2}}
\put(8,4){\line(0,1){2}}
\put(9,5){\line(0,1){1}}
\put(7,4){\line(1,0){1}}
\put(6,5){\line(1,0){3}}
\put(6,6){\line(1,0){3}}

\put(6,9){\line(0,1){1}}
\put(7,8){\line(0,1){2}}
\put(8,8){\line(0,1){2}}
\put(9,9){\line(0,1){1}}
\put(7,8){\line(1,0){1}}
\put(6,9){\line(1,0){3}}
\put(6,10){\line(1,0){3}}

\put(6,13){\line(0,1){1}}
\put(7,12){\line(0,1){2}}
\put(8,12){\line(0,1){2}}
\put(9,13){\line(0,1){1}}
\put(7,12){\line(1,0){1}}
\put(6,13){\line(1,0){3}}
\put(6,14){\line(1,0){3}}

\put(6,17){\line(0,1){1}}
\put(7,16){\line(0,1){2}}
\put(8,16){\line(0,1){2}}
\put(9,17){\line(0,1){1}}
\put(7,16){\line(1,0){1}}
\put(6,17){\line(1,0){3}}
\put(6,18){\line(1,0){3}}

\put(6,21){\line(0,1){1}}
\put(7,20){\line(0,1){2}}
\put(8,20){\line(0,1){2}}
\put(9,21){\line(0,1){1}}
\put(7,20){\line(1,0){1}}
\put(6,21){\line(1,0){3}}
\put(6,22){\line(1,0){3}}

\put(12,1){\line(0,1){1}}
\put(13,0){\line(0,1){2}}
\put(14,0){\line(0,1){2}}
\put(15,1){\line(0,1){1}}
\put(13,0){\line(1,0){1}}
\put(12,1){\line(1,0){3}}
\put(12,2){\line(1,0){3}}

\put(12,5){\line(0,1){1}}
\put(13,4){\line(0,1){2}}
\put(14,4){\line(0,1){2}}
\put(15,5){\line(0,1){1}}
\put(13,4){\line(1,0){1}}
\put(12,5){\line(1,0){3}}
\put(12,6){\line(1,0){3}}

\put(12,9){\line(0,1){1}}
\put(13,8){\line(0,1){2}}
\put(14,8){\line(0,1){2}}
\put(15,9){\line(0,1){1}}
\put(13,8){\line(1,0){1}}
\put(12,9){\line(1,0){3}}
\put(12,10){\line(1,0){3}}

\put(12,13){\line(0,1){1}}
\put(13,12){\line(0,1){2}}
\put(14,12){\line(0,1){2}}
\put(15,13){\line(0,1){1}}
\put(13,12){\line(1,0){1}}
\put(12,13){\line(1,0){3}}
\put(12,14){\line(1,0){3}}

\put(12,17){\line(0,1){1}}
\put(13,16){\line(0,1){2}}
\put(14,16){\line(0,1){2}}
\put(15,17){\line(0,1){1}}
\put(13,16){\line(1,0){1}}
\put(12,17){\line(1,0){3}}
\put(12,18){\line(1,0){3}}

\put(12,21){\line(0,1){1}}
\put(13,20){\line(0,1){2}}
\put(14,20){\line(0,1){2}}
\put(15,21){\line(0,1){1}}
\put(13,20){\line(1,0){1}}
\put(12,21){\line(1,0){3}}
\put(12,22){\line(1,0){3}}

\put(12,25){\line(0,1){1}}
\put(13,24){\line(0,1){2}}
\put(14,24){\line(0,1){2}}
\put(15,25){\line(0,1){1}}
\put(13,24){\line(1,0){1}}
\put(12,25){\line(1,0){3}}
\put(12,26){\line(1,0){3}}

\put(18,5){\line(0,1){1}}
\put(19,4){\line(0,1){2}}
\put(20,4){\line(0,1){2}}
\put(21,5){\line(0,1){1}}
\put(19,4){\line(1,0){1}}
\put(18,5){\line(1,0){3}}
\put(18,6){\line(1,0){3}}

\put(18,9){\line(0,1){1}}
\put(19,8){\line(0,1){2}}
\put(20,8){\line(0,1){2}}
\put(21,9){\line(0,1){1}}
\put(19,8){\line(1,0){1}}
\put(18,9){\line(1,0){3}}
\put(18,10){\line(1,0){3}}

\put(18,13){\line(0,1){1}}
\put(19,12){\line(0,1){2}}
\put(20,12){\line(0,1){2}}
\put(21,13){\line(0,1){1}}
\put(19,12){\line(1,0){1}}
\put(18,13){\line(1,0){3}}
\put(18,14){\line(1,0){3}}

\put(18,17){\line(0,1){1}}
\put(19,16){\line(0,1){2}}
\put(20,16){\line(0,1){2}}
\put(21,17){\line(0,1){1}}
\put(19,16){\line(1,0){1}}
\put(18,17){\line(1,0){3}}
\put(18,18){\line(1,0){3}}

\put(18,21){\line(0,1){1}}
\put(19,20){\line(0,1){2}}
\put(20,20){\line(0,1){2}}
\put(21,21){\line(0,1){1}}
\put(19,20){\line(1,0){1}}
\put(18,21){\line(1,0){3}}
\put(18,22){\line(1,0){3}}

\put(24,13){\line(0,1){1}}
\put(25,12){\line(0,1){2}}
\put(26,12){\line(0,1){2}}
\put(27,13){\line(0,1){1}}
\put(25,12){\line(1,0){1}}
\put(24,13){\line(1,0){3}}
\put(24,14){\line(1,0){3}}

\put(30,9){\line(0,1){1}}
\put(31,8){\line(0,1){2}}
\put(32,8){\line(0,1){2}}
\put(33,9){\line(0,1){1}}
\put(31,8){\line(1,0){1}}
\put(30,9){\line(1,0){3}}
\put(30,10){\line(1,0){3}}

\put(30,13){\line(0,1){1}}
\put(31,12){\line(0,1){2}}
\put(32,12){\line(0,1){2}}
\put(33,13){\line(0,1){1}}
\put(31,12){\line(1,0){1}}
\put(30,13){\line(1,0){3}}
\put(30,14){\line(1,0){3}}

\put(30,17){\line(0,1){1}}
\put(31,16){\line(0,1){2}}
\put(32,16){\line(0,1){2}}
\put(33,17){\line(0,1){1}}
\put(31,16){\line(1,0){1}}
\put(30,17){\line(1,0){3}}
\put(30,18){\line(1,0){3}}

\end{picture}
\end{center}
\caption{$\mathfrak{q}$(3)-crystal structure of $\mathrm{SSDT}_{3}(\lambda =(3,1))$.}
 \label{fig:SSDT}
\end{figure}

The $\mathfrak{q}(n)$-highest weight vector in $\mathrm{SSDT}_{n}(\lambda)$ is the semistandard decomposition tableau whose subtableau with entry $l(\lambda)-i+1$ is a connected border strip of size $\lambda_{l(\lambda)-i+1}$ starting at $T_{i,i}$ for $i=l(\lambda),\ldots,1$ and the $\mathfrak{q}(n)$-lowest weight vector is the one where the $i$-th row is filled with entry $n-i+1$ for $i=1,\ldots,l(\lambda)$~\cite{CK,GJKKK2}.

\begin{ex}
Let $n=4$ and $\lambda=(5,3,1)$.
Then the $\mathfrak{q}(n)$-highest weight vector (left) and the $\mathfrak{q}(n)$-lowest weight vector (right) are

\setlength{\unitlength}{12pt}

\begin{center}
\begin{picture}(13,3)

\put(0,2){\line(0,1){1}}
\put(1,1){\line(0,1){2}}
\put(2,0){\line(0,1){3}}
\put(3,0){\line(0,1){3}}
\put(4,1){\line(0,1){2}}
\put(5,2){\line(0,1){1}}
\put(2,0){\line(1,0){1}}
\put(1,1){\line(1,0){3}}
\put(0,2){\line(1,0){5}}
\put(0,3){\line(1,0){5}}

\put(2,0){\makebox(1,1){$1$}}
\put(1,1){\makebox(1,1){$2$}}
\put(2,1){\makebox(1,1){$1$}}
\put(3,1){\makebox(1,1){$1$}}
\put(0,2){\makebox(1,1){$3$}}
\put(1,2){\makebox(1,1){$2$}}
\put(2,2){\makebox(1,1){$2$}}
\put(3,2){\makebox(1,1){$1$}}
\put(4,2){\makebox(1,1){$1$}}

\put(8,2){\line(0,1){1}}
\put(9,1){\line(0,1){2}}
\put(10,0){\line(0,1){3}}
\put(11,0){\line(0,1){3}}
\put(12,1){\line(0,1){2}}
\put(13,2){\line(0,1){1}}
\put(10,0){\line(1,0){1}}
\put(9,1){\line(1,0){3}}
\put(8,2){\line(1,0){5}}
\put(8,3){\line(1,0){5}}

\put(10,0){\makebox(1,1){$2$}}
\put(9,1){\makebox(1,1){$3$}}
\put(10,1){\makebox(1,1){$3$}}
\put(11,1){\makebox(1,1){$3$}}
\put(8,2){\makebox(1,1){$4$}}
\put(9,2){\makebox(1,1){$4$}}
\put(10,2){\makebox(1,1){$4$}}
\put(11,2){\makebox(1,1){$4$}}
\put(12,2){\makebox(1,1){$4$}}

\end{picture}.
\end{center}

\end{ex}

\section{$\mathfrak{q}(n)$-crystal structure on primed tableaux} \label{sec:primed}

We denote by $\mathcal{B}_{n}^{m}$ the set of words consisting of $m$ letters from the alphabet $X_{n}$.
This set is a $\mathfrak{gl}(n)$-crystal and is also a $\mathfrak{q}(n)$-crystal.
The $\mathfrak{gl}(n)$ (or $\mathfrak{q}(n)$)-crystal structure of $\mathcal{B}_{n}^{m}$ is given by 
identifying $\boldsymbol{b}=b_{1}b_{2}\cdots b_{m} \in \mathcal{B}_{n}^{m}$ with 
the tensor product of $\mathfrak{gl}(n)$ (or $\mathfrak{q}(n)$)-crystal $B_{n}(\square)$, 
\setlength{\unitlength}{12pt}

\begin{center}
\begin{picture}(6,1)

\put(0,0){\framebox(1,1){$b_{1}$}}
\put(5,0){\framebox(1,1){$b_{m}$}}

\put(1,0){\makebox(1,1){$\otimes$}}
\put(4,0){\makebox(1,1){$\otimes$}}
\put(2,0){\makebox(2,1){$\cdots$}}

\end{picture}.
\end{center}

In practice, the tensor product rule for actions of $\Tilde{e}_{i}$ and $\Tilde{f}_{i}$ ($i=1,\ldots,n-1$) is rephrased by the following \emph{bracketing rule} on the word $\boldsymbol{b}$:
The Kashiwara operators $\Tilde{e}_{i}$ and $\Tilde{f}_{i}$ depend only on the letters $i$ and $i+1$ in $\boldsymbol{b}$.
We successively bracket any adjacent pairs $(i+1,i)$ and remove these pairs from the word.
We call this procedure the $(i+1,i)$ bracketing. 
If all letters $i+1$ are bracketed, then $\Tilde{e}_{i}\boldsymbol{b}=\boldsymbol{0}$.
Otherwise, $\Tilde{e}_{i}$ changes the leftmost unbracketed letter $i+1$ to $i$.
If all letters $i$ are bracketed, then $\Tilde{f}_{i}\boldsymbol{b}=\boldsymbol{0}$.
Otherwise, $\Tilde{f}_{i}$ changes the rightmost unbracketed letter $i$ to $i+1$.
The rules of action of odd Kashiwara operators $\Tilde{e}_{\Bar{1}}$ and $\Tilde{f}_{\Bar{1}}$ on the word $\boldsymbol{b}$ are quite simple.
If the letter $2$ does not exist in $\boldsymbol{b}$, then $\Tilde{e}_{\Bar{1}}\boldsymbol{b}=\boldsymbol{0}$.
Let $b_{k}$ be the leftmost $2$ in $\boldsymbol{b}$.
If the subword $b_{1}\cdots b_{k-1}$ contains the letter $1$, then $\Tilde{e}_{\Bar{1}}\boldsymbol{b}=\boldsymbol{0}$.
Otherwise, $\Tilde{e}_{\Bar{1}}$ changes $b_{k}=2$ to $1$.
If the letter $1$ does not exist in $\boldsymbol{b}$, then $\Tilde{f}_{\Bar{1}}\boldsymbol{b}=\boldsymbol{0}$.
Let $b_{k}$ be the leftmost $1$ in $\boldsymbol{b}$.
If the subword $b_{1}\cdots b_{k-1}$ contains the letter $2$, then $\Tilde{f}_{\Bar{1}}\boldsymbol{b}=\boldsymbol{0}$.
Otherwise, $\Tilde{f}_{\Bar{1}}$ changes $b_{k}=1$ to $2$.
One can easily check that the conditions (3), (4), and (5) in Definition~\ref{df:queer} are satisfied, 
where the weight of $\boldsymbol{b}\in \mathcal{B}_{n}^{m}$ is $\mathrm{wt}(\boldsymbol{b})=\sum _{i=1}^{n}a_{i}\epsilon _{i}$ 
with $a_{i}$ being the number of letters $i$ in $\boldsymbol{b}$.

As an element of the $\mathfrak{gl}(n)$-crystal, a word $\boldsymbol{b}$ is highest weight, i.e., $\Tilde{e}_{i}\boldsymbol{b}=\boldsymbol{0}$ for all $i=1,2,\ldots,n-1$, if and only if it is a Yamanouchi word~\cite{KN} .
That is, the weight of the subword $b_{k}b_{k+1}\cdots b_{m}$ is a partition for all $k=1,2,\ldots,m$.
However, Yamanouchi words are not always highest weight in a sense of $\mathfrak{q}(n)$-crystals.
For example, the word $\boldsymbol{b}=321121$ is Yamanouchi, but $\Tilde{e}_{\Bar{1}}\boldsymbol{b}=311121\neq\boldsymbol{0}$.

\begin{df}
A shifted tableau $T$ with $n$ boxes is called a \emph{standard shifted Young tableau} if 
\begin{itemize}
\item[(1)]
the entries are all distinct ranging from $1$ to $n$,
\item[(2)]
the numbers in each rows are strictly increasing, and 
\item[(3)]
the numbers in each column are strictly increasing.
\end{itemize}
We denote by $\mathrm{ST}_{n}(\lambda)$ the set of standard shifted Young tableaux with $n$ boxes and shape $\lambda$. 
\end{df}

\begin{df}
A \emph{primed tableau} $T$ of shape $\lambda$ is a filling of $S(\lambda)$ with letters from the alphabet 
$X_{n}^{\prime}:=\{1^{\prime}<1<2^{\prime}<2<\cdots<n\}$ such that: 
\begin{itemize}
\item[(1)] the entries are weakly increasing along each column and each row of $T$,
\item[(2)] each row contains at most one $i^{\prime}$ for every $i=1,\ldots,n$,
\item[(3)] each column contains at most one $i$ for every $i=1,\ldots,n$, and
\item[(4)] there are no primed letters on the main diagonal.
\end{itemize}
We denote by $\mathrm{PT}_{n}(\lambda)$ the set of primed tableaux of shape $\lambda$. 
\end{df}

%Note that the entry $1^{\prime}$ does not appear in $T\in \mathrm{PT}_{n}(\lambda)$.

\begin{rem}
A primed tableau is also called a semistandard marked shifted tableau~\cite{Cho}.
\end{rem}

For a word $\boldsymbol{b}=b_{1}\cdots b_{m}\in \mathcal{B}_{n}^{m}$ in the alphabet $X_{n}$, 
we recursively construct a sequence of pairs of tableaux, 
\[
(\emptyset,\emptyset)=(T^{(0)},Q^{(0)}),(T^{(1)},Q^{(1)}),\ldots,(T^{(m)},Q^{(m)})=(T,Q),
\]
where 
$T^{(i)}\in \mathrm{PT}_{n}(\lambda^{(i)})$ and 
$Q^{(i)}\in \mathrm{ST}_{n}(\lambda^{(i)})$ by the insertion algorithm due to Serrano~\cite{Ser}.
The algorithm to obtain $(T^{(i)},Q^{(i)})$ from $(T^{(i-1)},Q^{(i-1)})$ is as follows:
\begin{itemize}
\item[(1)]
Insert $b_{i}$ into the first row of $T^{(i-1)}$, bumping out the leftmost entry $a$ that is strictly greater than $b_{i}$.
If such an $a$ does not exist, append \framebox[12pt]{$b_{i}$} to the end of the first row of $T^{(i-1)}$ and 
let $T^{(i)}$ be this new tableau.
To get $Q^{(i)}$, we add \framebox[12pt]{$i$} to $Q^{(i-1)}$ such that $T^{(i)}$ and $Q^{(i)}$ have the same shifted shape.
Stop.
If such an $a$ exists, do as follows:

\item[(2)]
\begin{itemize}
\item[(i):]
$a$ is not on the main diagonal.

\begin{itemize}
\item[(i-1)]
If $a$ is not primed, insert it into the next row, bumping out the leftmost entry that is strictly greater than $a$ from that row $R$.

\item[(i-2)]
If $a$ is primed, insert it into the next column to the right, bumping out the topmost entry that is strictly greater than $a$ from that column $C$.
\end{itemize}

\item[(ii):]
$a$ is on the main diagonal.
\begin{itemize}
\item[]
Prime it and insert it into the column to the right, bumping out the topmost entry that is strictly greater than $a$ from that column $C$.
\end{itemize}
\end{itemize}

\item[(3)]
If the bumped element does not exist, append \framebox[12pt]{\vphantom{$i$}$a$} to the end of the row $R$ (in case (i-1)) 
or the column $C$ (in case (i-2) or (ii)), and let $T^{(i)}$ be this new tableau.
To get $Q^{(i)}$, we add \framebox[12pt]{$i$} to $Q^{(i-1)}$ such that $T^{(i)}$ and $Q^{(i)}$ have the same shifted shape.
Stop.
If the bumped element exists, repeat (2) and (3) with $a$ being this bumped element.
\end{itemize}

We call $T$ (resp. $Q$) above the \emph{mixed insertion} (resp. \emph{recording}) tableau.
This algorithm, the \emph{semistandard shifted mixed insertion}, is a semistandard generalization of the mixed insertion introduced by Haiman~\cite{Hai1}.

\begin{ex}
The word $\boldsymbol{b}=333323212$ has the following mixed insertion (left) and recording (right) tableaux.

\setlength{\unitlength}{12pt}

\begin{center}
\begin{picture}(13,3)

\put(0,2){\line(0,1){1}}
\put(1,1){\line(0,1){2}}
\put(2,0){\line(0,1){3}}
\put(3,0){\line(0,1){3}}
\put(4,1){\line(0,1){2}}
\put(5,2){\line(0,1){1}}
\put(2,0){\line(1,0){1}}
\put(1,1){\line(1,0){3}}
\put(0,2){\line(1,0){5}}
\put(0,3){\line(1,0){5}}

\put(2,0){\makebox(1,1){$3$}}
\put(1,1){\makebox(1,1){$2$}}
\put(2,1){\makebox(1,1){$3^{\prime}$}}
\put(3,1){\makebox(1,1){$3$}}
\put(0,2){\makebox(1,1){$1$}}
\put(1,2){\makebox(1,1){$2^{\prime}$}}
\put(2,2){\makebox(1,1){$2$}}
\put(3,2){\makebox(1,1){$3^{\prime}$}}
\put(4,2){\makebox(1,1){$3$}}

\put(8,2){\line(0,1){1}}
\put(9,1){\line(0,1){2}}
\put(10,0){\line(0,1){3}}
\put(11,0){\line(0,1){3}}
\put(12,1){\line(0,1){2}}
\put(13,2){\line(0,1){1}}
\put(10,0){\line(1,0){1}}
\put(9,1){\line(1,0){3}}
\put(8,2){\line(1,0){5}}
\put(8,3){\line(1,0){5}}

\put(10,0){\makebox(1,1){$8$}}
\put(9,1){\makebox(1,1){$5$}}
\put(10,1){\makebox(1,1){$7$}}
\put(11,1){\makebox(1,1){$9$}}
\put(8,2){\makebox(1,1){$1$}}
\put(9,2){\makebox(1,1){$2$}}
\put(10,2){\makebox(1,1){$3$}}
\put(11,2){\makebox(1,1){$4$}}
\put(12,2){\makebox(1,1){$6$}}

\end{picture}.
\end{center}

\end{ex}

\begin{thm}[\cite{HPS,Ser}]
The semistandard shifted mixed insertion gives a bijection 
\begin{equation} \label{eq:HM}
\mathrm{HM}:\mathcal{B}_{n}^{m}\stackrel{\sim}{\longrightarrow}
%TCIMACRO{\tbigcup \limits_{\lambda\vdash h}}%
%BeginExpansion
{\textstyle\bigcup\limits_{\lambda\vdash m}}
%EndExpansion
\left[  \mathrm{PT}_{n}(\lambda)\times\mathrm{ST}_{n}(\lambda)\right].
\end{equation}
\end{thm}

For $\boldsymbol{b}\in \mathcal{B}_{n}^{m}$, we write $\mathrm{HM}(\boldsymbol{b})=(T,Q)$ and $P_{HM}(\boldsymbol{b})=T$.
%and $R_{HM}(\boldsymbol{b})=Q$. 
By fixing a mixed recording tableau $Q_{\lambda}\in \mathrm{ST}_{n}(\lambda)$, 
one can define a map 
$\Psi _{\lambda}: \mathrm{PT}_{n}(\lambda) \rightarrow \mathcal{B}_{n}^{m}$ as 
$\Psi_{\lambda}(T)=\mathrm{HM}^{-1}(T,Q_{\lambda})$.
Utilizing this map, Hawkes, Paramonov, and Schilling constructed Kashiwara operators $\Tilde{e}_{i}^{P}$ and $\Tilde{f}_{i}^{P}$ on $T\in \mathrm{PT}_{n}(\lambda)$ and 
showed that they are independent of the choice of $Q_{\lambda}$ by giving the explicit algorithms for $\Tilde{e}_{i}^{P}$ and $\Tilde{f}_{i}^{P}$ ($i=1,\ldots,n-1$)~\cite{HPS}.
Since the set of words $\mathcal{B}_{n}^{m}$ admits a $\mathfrak{q}(n)$-crystal structure, $\mathrm{PT}_{n}(\lambda)$ is also a $\mathfrak{q}(n)$-crystal.
This can be also seen by the bijection between the set of semistandard unimodal tableaux of shape $\lambda$ (see Remark~\ref{rem:plethora}), denoted by $\mathrm{SUT}(\lambda)$, and $\mathrm{PT}(\lambda)$~\cite{HPS} and 
by the one between $\mathrm{SUT}(\lambda)$ and $\mathrm{SSDT}(\lambda)$~\cite{CNO}.
In particular, all the elements in $\mathrm{PT}_{n}(\lambda)$ are connected by the even and odd Kashiwara operators.
The odd Kashiwara operators $\Tilde{e}_{\Bar{1}}^{P}$ and $\Tilde{f}_{\Bar{1}}^{P}$ on $\mathrm{PT}_{n}(\lambda)$ are constructed in a similar way as in \cite{HPS}.

%Let us state our first main result.
\begin{thm} \label{thm:primed}
The set of primed tableaux $\mathrm{PT}_{n}(\lambda)$ admits a $\mathfrak{q}(n)$-crystal structure by the bijection $\mathrm{HM}$ (Eq.~\eqref{eq:HM}) through the $\mathfrak{q}(n)$-crystal structure of $\mathcal{B}_{n}^{m}$.
In particular, odd Kashiwara operators $\Tilde{e}_{\Bar{1}}^{P}$ and $\Tilde{f}_{\Bar{1}}^{P}$ as well as 
even Kashiwara operators $\Tilde{e}_{i}^{P}$ and $\Tilde{f}_{i}^{P}$ ($i=1,\ldots,n-1$) are independent of the choice of $Q_{\lambda}\in \mathrm{ST}_{n}(\lambda)$ in Eq.~\eqref{eq:HM}.
\end{thm}

\begin{proof}
Since even Kashiwara operators were shown to be independent of the choice of $Q_{\lambda}\in \mathrm{ST}_{n}(\lambda)$ in Eq.~\eqref{eq:HM}, 
it remains to prove that odd Kashiwara operators $\Tilde{e}_{\Bar{1}}^{P}$ and $\Tilde{f}_{\Bar{1}}^{P}$ are also independent of the choice of $Q_{\lambda}$.
The claim is the consequence of the following two lemmas (Lemma~\ref{lem:eP} and \ref{lem:fP}).
\end{proof}

The explicit description of the operators $\Tilde{e}_{\Bar{1}}^{P}$ and $\Tilde{f}_{\Bar{1}}^{P}$ resulting from Lemma~\ref{lem:eP} and \ref{lem:fP} is the same as in Definition 4.5 in \cite{AO2}.
Nevertheless, we present the following two lemmas with proofs because the way of construction is completely different from that in \cite{AO1,AO2}.

\begin{lem} \label{lem:eP}
For $T\in\mathrm{PT}_{n}(\lambda)$, the odd Kashiwara operator $\Tilde{e}_{\Bar{1}}^{P}$ on $T$ is given by the following rule: 

If $T_{1,1}=2$, change $T_{1,1}$ to $1$.
If $T_{1,i}=2^{\prime}$ ($i\geq 2$), change $T_{1,i}$ to $1$.
Otherwise, $\Tilde{e}_{\Bar{1}}^{P}T=\boldsymbol{0}$.
\end{lem}

Note that the action of $\Tilde{e}_{\Bar{1}}^{P}$ on $T\in\mathrm{PT}_{n}(\lambda)$ depends only on $T$.

\begin{proof}

Let $\boldsymbol{b}=b_{1}b_{2}\cdots b_{m}\in \mathcal{B}_{n}^{m}$ be the word such that 
$\mathrm{HM}(\boldsymbol{b})=(T,Q_{\lambda})$
($P_{HM}(\boldsymbol{b})=T$ for a fixed $Q_{\lambda}$).
Let $b_{k}$ be the leftmost $2$ in $\boldsymbol{b}$ ($1\leq k\leq m$).
In the following, we assume that the subword $b_{1}b_{2}\ldots b_{k-1}$ does not contain $1$; 
otherwise $\Tilde{e}_{\Bar{1}}\boldsymbol{b}=\boldsymbol{0}$.
Suppose that the number of letters $1$ in the subword $b_{k+1}b_{k+2}\cdots b_{m}$ is $i-1$.
\begin{itemize}
\item[(1):]
$i=1$.
The first row of $T$ has the following configuration.
\setlength{\unitlength}{12pt}

\begin{center}
\begin{picture}(4,1)

\put(0,0){\line(0,1){1}}
\put(1,0){\line(0,1){1}}
\put(2,0){\line(0,1){1}}
\put(0,0){\line(1,0){4}}
\put(0,1){\line(1,0){4}}

\put(0,0){\makebox(1,1){$2$}}
\put(1,0){\makebox(1,1){$x$}}
\put(2,0){\makebox(2,1){$\cdots$}}

\end{picture},
\end{center}
where $x\geq 2$.

\item[(2):]
$i\geq 2$.
The first row of $T$ has the following configuration.
\setlength{\unitlength}{12pt}

\begin{center}
\begin{picture}(7,2)

\put(0,0){\line(0,1){1}}
\put(1,0){\line(0,1){1}}
\put(3,0){\line(0,1){1}}
\put(4,0){\line(0,1){1}}
\put(5,0){\line(0,1){1}}
\put(0,0){\line(1,0){7}}
\put(0,1){\line(1,0){7}}

\put(0,0){\makebox(1,1){$1$}}
\put(1,0){\makebox(2,1){$\cdots$}}
\put(3,0){\makebox(1,1){$1$}}
\put(4,0){\makebox(1,1){$2^{\prime}$}}
\put(5,0){\makebox(2,1){$\cdots$}}

\put(4,1){\makebox(1,1){\small $i$}}
\end{picture}.
\end{center}

\end{itemize}
Since 
$\Tilde{e}_{\Bar{1}}\boldsymbol{b}=b_{1}\cdots b_{k-1}1b_{k+1}\cdots b_{m}=\boldsymbol{b}^{\prime}$, 
the first row of $T^{\prime}=P_{HM}(\boldsymbol{b}^{\prime})$ has the configuration,
\setlength{\unitlength}{12pt}

\begin{center}
\begin{picture}(4,1)

\put(0,0){\line(0,1){1}}
\put(1,0){\line(0,1){1}}
\put(2,0){\line(0,1){1}}
\put(0,0){\line(1,0){4}}
\put(0,1){\line(1,0){4}}

\put(0,0){\makebox(1,1){$1$}}
\put(1,0){\makebox(1,1){$x$}}
\put(2,0){\makebox(2,1){$\cdots$}}

\end{picture}
\end{center}
for Case (1) and 
\setlength{\unitlength}{12pt}

\begin{center}
\begin{picture}(7,2)

\put(0,0){\line(0,1){1}}
\put(1,0){\line(0,1){1}}
\put(3,0){\line(0,1){1}}
\put(4,0){\line(0,1){1}}
\put(5,0){\line(0,1){1}}
\put(0,0){\line(1,0){7}}
\put(0,1){\line(1,0){7}}

\put(0,0){\makebox(1,1){$1$}}
\put(1,0){\makebox(2,1){$\cdots$}}
\put(3,0){\makebox(1,1){$1$}}
\put(4,0){\makebox(1,1){$1$}}
\put(5,0){\makebox(2,1){$\cdots$}}

\put(4,1){\makebox(1,1){\small $i$}}
\end{picture}
\end{center}
for Case (2).
Thus, the proof is complete.

It is clear from the above arguments that the odd Kashiwara operator $\Tilde{e}_{\Bar{1}}^{P}$ does not depend on the choice of the recording tableau $Q_{\lambda}$ in Eq.~\eqref{eq:HM} and that 
$P_{HM}(\Tilde{e}_{\Bar{1}}\boldsymbol{b})=\Tilde{e}_{\Bar{1}}^{P}T$.

\end{proof}

\begin{lem} \label{lem:fP}
For $T\in\mathrm{PT}_{n}(\lambda)$, the odd Kashiwara operator $\Tilde{f}_{\Bar{1}}^{P}$ on $T$ is given by the following rule:

If the first row of $T$ does not contain letters $1$, then $\Tilde{f}_{\Bar{1}}^{P}T=\boldsymbol{0}$.
Let $(1,i)$ be the position of the rightmost $1$ in the first row.
\begin{itemize}
\item[(1):]
$i=1$.
If $T_{1,2}=2^{\prime}$, then $\Tilde{f}_{\Bar{1}}^{P}T=\boldsymbol{0}$.
Otherwise, change $T_{1,1}=1$ to $2$.
\item[(2):]
$i\geq 2$.
If $T_{1,i+1}=2^{\prime}$, then $\Tilde{f}_{\Bar{1}}^{P}T=\boldsymbol{0}$.
Otherwise, change $T_{1,i}=1$ to $2^{\prime}$.
\end{itemize}
\end{lem}

Again, note that the action of $\Tilde{f}_{\Bar{1}}^{P}$ on $T\in\mathrm{PT}_{n}(\lambda)$ depends only on $T$.

\begin{proof}

Let $\boldsymbol{b}=b_{1}b_{2}\cdots b_{m}\in \mathcal{B}_{n}^{m}$ such that $\mathrm{HM}(\boldsymbol{b})=(T,Q_{\lambda})$ 
($P_{HM}(\boldsymbol{b})=T$ for a fixed $Q_{\lambda}$).
Let $b_{k}$ be the leftmost $1$ in $\boldsymbol{b}$ ($1\leq k\leq m$).
In the following, we assume that the subword $b_{1}b_{2}\cdots b_{k-1}$ does not contain $2$; 
otherwise $\Tilde{f}_{\Bar{1}}\boldsymbol{b}=\boldsymbol{0}$ .
Suppose that the number of letters $1$ in the subword $b_{k+1}b_{k+2}\cdots b_{m}$ is $i-1$.
\begin{itemize}
\item[(1):]
$i=1$.
The first row of $T$ has the following configuration.
\setlength{\unitlength}{12pt}

\begin{center}
\begin{picture}(4,1)

\put(0,0){\line(0,1){1}}
\put(1,0){\line(0,1){1}}
\put(2,0){\line(0,1){1}}
\put(0,0){\line(1,0){4}}
\put(0,1){\line(1,0){4}}

\put(0,0){\makebox(1,1){$1$}}
\put(1,0){\makebox(1,1){$x$}}
\put(2,0){\makebox(2,1){$\cdots$}}

\end{picture}.
\end{center}
Here, $x\geq 2$.
This is shown as follows for $k\geq 3$ (for $k=1,2$ this is obvious).
Before inserting $b_{k}=1$, entries $T_{1,1}^{(k-1)}$ (unprimed) and $T_{1,2}^{(k-1)}$ are greater than or equal to $3$ 
because $2\notin \{b_{1},\ldots,b_{k-1}\}$.
Then, after inserting $b_{k}=1$, $T_{1,1}^{(k)}=1$ and $T_{1,2}^{(k)}\geq 3^{\prime}$. 
Subsequent insertion of a letter $2$ (if it exists) changes $T_{1,2}^{(k)}$ to $2$ but not to $2^{\prime}$ due to the step (1) in the semistandard shifted mixed insertion.

\item[(2):]
$i\geq 2$.
The first row of $T$ has the following configuration.
\setlength{\unitlength}{12pt}

\begin{center}
\begin{picture}(7,2)

\put(0,0){\line(0,1){1}}
\put(1,0){\line(0,1){1}}
\put(3,0){\line(0,1){1}}
\put(4,0){\line(0,1){1}}
\put(5,0){\line(0,1){1}}
\put(0,0){\line(1,0){7}}
\put(0,1){\line(1,0){7}}

\put(0,0){\makebox(1,1){$1$}}
\put(1,0){\makebox(2,1){$\cdots$}}
\put(3,0){\makebox(1,1){$1$}}
\put(4,0){\makebox(1,1){$x$}}
\put(5,0){\makebox(2,1){$\cdots$}}

\put(3,1){\makebox(1,1){\small $i$}}
\end{picture}.
\end{center}
Here,  $x\geq 2$ as in (1).
\end{itemize}
In both cases, $x\neq 2^{\prime}$.
Since $\Tilde{f}_{\Bar{1}}\boldsymbol{b}=b_{1}\cdots b_{k-1}2b_{k+1}\cdots b_{m}=\boldsymbol{b}^{\prime}$, 
the first row of $T^{\prime}=P_{HM}(\boldsymbol{b}^{\prime})$ has the configuration,
\setlength{\unitlength}{12pt}

\begin{center}
\begin{picture}(4,1)

\put(0,0){\line(0,1){1}}
\put(1,0){\line(0,1){1}}
\put(2,0){\line(0,1){1}}
\put(0,0){\line(1,0){4}}
\put(0,1){\line(1,0){4}}

\put(0,0){\makebox(1,1){$2$}}
\put(1,0){\makebox(1,1){$x$}}
\put(2,0){\makebox(2,1){$\cdots$}}

\end{picture}
\end{center}
for Case(1) and 
\setlength{\unitlength}{12pt}

\begin{center}
\begin{picture}(7,2)

\put(0,0){\line(0,1){1}}
\put(1,0){\line(0,1){1}}
\put(3,0){\line(0,1){1}}
\put(4,0){\line(0,1){1}}
\put(5,0){\line(0,1){1}}
\put(0,0){\line(1,0){7}}
\put(0,1){\line(1,0){7}}

\put(0,0){\makebox(1,1){$1$}}
\put(1,0){\makebox(2,1){$\cdots$}}
\put(3,0){\makebox(1,1){$2^{\prime}$}}
\put(4,0){\makebox(1,1){$x$}}
\put(5,0){\makebox(2,1){$\cdots$}}

\put(3,1){\makebox(1,1){\small $i$}}
\end{picture}
\end{center}
for Case (2).
In both cases, $x\neq 2^{\prime}$; otherwise $\Tilde{f}_{\Bar{1}}^{P}T \notin \mathrm{PT}_{n}(\lambda)$, i.e., 
$\Tilde{f}_{\Bar{1}}^{P}T=\boldsymbol{0}$.

It is clear from the above arguments that the odd Kashiwara operator 
 $\Tilde{f}_{\Bar{1}}^{P}$ does not depend on the choice of the recording tableau $Q_{\lambda}$ in Eq.~\eqref{eq:HM} and that 
$P_{HM}(\Tilde{f}_{\Bar{1}}\boldsymbol{b})=\Tilde{f}_{\Bar{1}}^{P}T$.

\end{proof}

%In the rest of this section, we will give the explicit odd Kashiwara operators on $\mathrm{PT}_{n}(\lambda)$ and 
%will show that they are independent of the choice of $Q_{\lambda}$.
We will also give the form of the $\mathfrak{q}(n)$-highest weight vector as well as $\mathfrak{q}(n)$-lowest weight vector of $\mathrm{PT}_{n}(\lambda)$ (Theorem~\ref{thm:highlow}).
To construct the $\mathfrak{q}(n)$-highest weight vector, we need the explicit rules of $\Tilde{f}_{i}^{P}$ on $T\in \mathrm{PT}_{n}(\lambda)$ described below~\cite{HPS}.
Another but equivalent rule is given in \cite{AO1,AO2}.

We construct the reading word $\mathrm{rw}(T)$ of $T$ as follows:
We first read primed letters through each column from top to bottom staring from the rightmost column to the leftmost one, 
and remove all the primes. 
Next, we read unprimed letters through each row from left to right starting from the bottommost row to the topmost one.
The weight of $T$ is defined to be $\mathrm{wt}(T)=\sum _{i=1}^{n}a_{i}\epsilon _{i}$, 
where $a_{i}$ is the number of letters $i$ in $\mathrm{rw}(T)$.
We also write $\mathrm{wt}(T)=(a_{1}, a_{2},\ldots,a_{n})$.

We first apply the $(i+1,i)$-bracketing on the reading word $\mathrm{rw}(T)$.
If all letters $i$ are bracketed in $\mathrm{rw}(T)$, then $\Tilde{f}_{i}^{P}T=\boldsymbol{0}$.
Otherwise, the rightmost unbracketed letter $i$ in $\mathrm{rw}(T)$ corresponding to an $i$ or an $i^{\prime}$ in $T$, which we call bold unprimed $i$ or bold primed $i$ respectively.

\begin{itemize}
\item[(1)] If the bold letter $i$ is unprimed, denote the cell it is located in as $x$.
\item[(2)] If the bold letter $i$ is primed, conjugate the tableau $T$ first.
The conjugation of a primed tableau $T$ is obtained by reflecting the tableau over the main diagonal, changing all primed entries $k^{\prime}$ to $k$ and changing all umprimed entries $k$ to $(k+1)^{\prime}$.
Denote the resulting tableau as $T^{\ast}$.
The bold primed $i$ in $T$ is now the bold unprimed $i$ in $T^{\ast}$,
Denote the cell it is locates in as $x$.
\end{itemize}

Given any cell $z$ in a signed primed tableau $T$ (or its conjugation $T^{\ast}$), denote by $c(z)$ the content of $z$ and 
by $z_{E}$ the cell to the right of $z$.
 
The algorithm of action of $\Tilde{f}_{i}^{P}$ on $T$ is as follows:
\begin{itemize}
\item[(1)]
If $c(x_{E})=(i+1)^{\prime}$, change $c(x)$ to $(i+1)^{\prime}$ and change $c(x_{E})$ to $(i+1)$.
\item[(2)]
If $c(x_{E})\neq (i+1)^{\prime}$ or $x_{E}$ is empty, then there is a maximal connected ribbon (expanding in South and West directions) with the following properties:
\begin{itemize}
\item[(i)]
The North-Eastern most box of the ribbon (the tail of the ribbon) is $x$.
\item[(ii)]
The contents of all boxes within a ribbon besides the tail are either $(i+1)^{\prime}$ or $(i+1)$.
\end{itemize}
Denote the South-Western most box of the ribbon (the head) as $x_{H}$.
\begin{itemize}
\item[(2-a)]
If $x_{H}=x$, change $c(x)$ to $(i+1)$.
\item[(2-b)]
If $x_{H}\neq x$ and $x_{H}$ is on the main diagonal (in case of a tableau $T$), change $c(x)$ to $(i+1)^{\prime}$.
\item[(2-c)]
Otherwise, change $c(x)$ to $(i+1)^{\prime}$ and change $c(x_{H})$ to $(i+1)$.
\end{itemize}

\end{itemize}

In the case when the bold $i$ is unprimed, we apply the above rules on $T$ to find $\Tilde{f}_{i}^{P}T$.
In the case when the bold $i$ is primed, we first conjugate $T$ and then apply the above rules on $T^{\ast}$, before reversing the conjugation.

\begin{lem}[\cite{HPS}] \label{lem:Yamanouchi}
As an element of a $\mathfrak{gl}(n)$-crystal $T\in \mathrm{PT}_{n}(\lambda)$ is a highest weight element, i.e., $\Tilde{e}_{i}^{P}T=\boldsymbol{0}$ for all $i=1,\ldots,n-1$, if and only if its reading word $\mathrm{rw}(T)$ is a Yamanouchi word.
\end{lem}

%The following is our second main result.
\begin{thm} \label{thm:highlow}
Let us consider $\mathrm{PT}_{n}(\lambda)$ as a $\mathfrak{q}(n)$-crystal.
The $\mathfrak{q}(n)$-highest weight vector in $\mathrm{PT}_{n}(\lambda)$ is a tableau 
whose $i$-th row is filled with letters $i$ and 
the $\mathfrak{q}(n)$-lowest weight vector in $\mathrm{PT}_{n}(\lambda)$ is a tableau where 
the subtableau with entries $n-l(\lambda)+i$ or $(n-l(\lambda)+i)^{\prime}$ is a connected border strip of size 
$\lambda_{l(\lambda)-i+1}$ starting at $T_{i,i}=n-l(\lambda)+i$ for $i=l(\lambda),\ldots,1$.
\end{thm}

Note that the entries in the subtableau in the $\mathfrak{q}(n)$-lowest weight vector are uniquely determined.

\begin{proof}

Let $T\in \mathrm{PT}_{n}(\lambda)$ be a primed tableau whose $i$-th row is filled with letters $i$ ($i=1.\ldots,l(\lambda)$).
It is clear by Lemma~\ref{lem:Yamanouchi} that $\Tilde{e}_{i}^{P}T=\boldsymbol{0}$ ($i=1,\ldots,n-1$) because the reading word $\mathrm{rw}(T)$ is a Yamanouchi word.
It is also obvious that $\Tilde{e}_{\Bar{1}}^{P}T=\boldsymbol{0}$ by Lemma~\ref{lem:eP}.

Let $i\in \{2,\ldots,n-1\}$ and let us write $\mathrm{wt}(T)=(\mu_{1},\mu_{2},\ldots,\mu_{n})$.
Then, $\mu_{1}=\lambda_{1},\ldots,\mu_{l}=\lambda_{l},\mu_{l+1}=\cdots=\mu_{n}=0$, where $l=l(\lambda)$.

Firstly, we determine the configuration of the tableau $S_{1}\cdots S_{i-1}T$.
Let us write $\mathrm{wt}(S_{i-1}T)=(\mu_{1}^{(i-1)},\mu_{2}^{(i-1)},\ldots,\mu_{n}^{(i-1)})$.
Since $S_{i-1}T=\left(  \Tilde{f}_{i-1}^{P}\right)  ^{(\mu_{i-1}-\mu_{i})}T=
\left(  \Tilde{f}_{i-1}^{P}\right)  ^{(\lambda_{i-1}-\lambda_{i})}T$, 
we have 
$\mu_{i-1}^{(i-1)}=\lambda_{i-1}-(\lambda_{i-1}-\lambda_{i})=\lambda_{i}$ and  
$\mu_{i}^{(i-1)}=\lambda_{i}+(\lambda_{i-1}-\lambda_{i})=\lambda_{i-1}$.
The other components of $\mathrm{wt}(S_{i-1}T)$ are the same as those of $\mathrm{wt}(T)$.
Set $T^{(k)}=S_{k}\cdots S_{i-1}T$ and 
$\mathrm{wt}(T^{(k)})=(\mu_{1}^{(k)},\mu_{2}^{(k)},\ldots,\mu_{n}^{(k)})$ for $k=i-1,\ldots,2$.
Assume that 
$\mu_{k}^{(k)}=\lambda_{i}, \mu_{k+1}^{(k)}=\lambda_{k}, \ldots, 
\mu_{i}^{(k)}=\lambda_{i-1}$, and that the other components of 
$\mathrm{wt}(T^{(k)})$ are the same as those of $\mathrm{wt}(T)$.
This assumption is satisfied for $k=i-1$.
Let $\mathrm{wt}(T^{(k-1)})=\mathrm{wt}(S_{k-1}T^{(k)})$ be 
$(\mu_{1}^{(k-1)},\mu_{2}^{(k-1)},\ldots,\mu_{n}^{(k-1)})$.
Since 
$S_{k-1}T^{(k)}=\left( \Tilde{f}_{k-1}^{P}\right)  ^{(\mu_{k-1}^{(k)}-\mu_{k}^{(k)})}T^{(k)}=
\left( \Tilde{f}_{k-1}^{P}\right)  ^{(\lambda_{k-1}-\lambda_{i})}T^{(k)}$, we have that 
$\mu_{k-1}^{(k-1)}=\lambda_{k-1}-(\lambda_{k-1}-\lambda_{i})=\lambda_{i}$, 
$\mu_{k}^{(k-1)}=\lambda_{i}+(\lambda_{k-1}-\lambda_{i})=\lambda_{k-1}$, 
and the other components of $\mathrm{wt}(S_{k-1}T^{(k)})$ are the same as those of $\mathrm{wt}(T^{(k)})$.
That is, 
$\mu_{k-1}^{(k-1)}=\lambda_{i}, \mu_{k}^{(k-1)}=\lambda_{k-1}, \ldots, 
\mu_{i}^{(k-1)}=\lambda_{i-1}$, and the other components of 
$\mathrm{wt}(T^{(k-1)})$ are the same as those of $\mathrm{wt}(T)$.
Hence, by induction, we have 
$\mathrm{wt}(T^{(1)})=\mathrm{wt}(S_{1}\cdots S_{i-1}T)=
(\mu_{1}^{(1)},\mu_{2}^{(1)}\ldots,\mu_{n}^{(1)})=
(\lambda_{i},\lambda_{1},\ldots,\lambda_{i-2},\lambda_{i-1},\lambda_{i+1},\ldots)$.
The weight of this form determines the configuration of $T^{(1)}$ uniquely.
The configuration of the $(k-1)$-th and $k$-th rows is
\setlength{\unitlength}{12pt}

\begin{center}
\begin{picture}(16,6)

\put(0,3){\line(0,1){1}}
\put(2,2){\line(0,1){2}}
\put(3,2){\line(0,1){1}}
\put(5,3){\line(0,1){1}}
\put(7,2){\line(0,1){2}}
\put(8,2){\line(0,1){2}}
\put(11,2){\line(0,1){2}}
\put(13,2){\line(0,1){1}}
\put(15,3){\line(0,1){1}}
\put(16,3){\line(0,1){1}}

\put(2,2){\line(1,0){13}}
\put(0,3){\line(1,0){16}}
\put(0,4){\line(1,0){16}}

\put(0,3){\makebox(2,1){$k-1$}}
\put(2,2){\makebox(1,1){$k$}}
\put(5,3){\makebox(2,1){$k-1$}}
\put(7,2){\makebox(1,1){$k$}}
\put(7,3){\makebox(1,1){$k^{\prime}$}}
\put(8,2){\makebox(3,1){$(k+1)^{\prime}$}}
\put(8,3){\makebox(3,1){$k$}}
\put(11,2){\makebox(2,1){$k+1$}}
\put(15,3){\makebox(1,1){$k$}}

\put(3,2){\makebox(4,1){$\cdots$}}
\put(2,3){\makebox(3,1){$\cdots$}}
\put(13,2){\makebox(2,1){$\cdots$}}
\put(11,3){\makebox(4,1){$\cdots$}}

\put(8,0){\makebox(3,1){$\vdots$}}
\put(8,5){\makebox(3,1){$\vdots$}}

\put(7,0){\makebox(1,1){\small $\lambda_{i}$}}
\put(7,1){\makebox(1,1){$\uparrow$}}
\end{picture},
\end{center}
for $k=2,\ldots,i-2$, where the rightmost $k$ in the $k$-th row is located in the $\lambda_{i}$-th box from the left, i.e., 
$T_{k,k+\lambda_{i}-1}^{(1)}=k$ and $T_{k,k+\lambda_{i}}^{(1)}=(k+1)^{\prime}$ and the all boxes of the right part from the $(\lambda_{i}+2)$-th position in the $k$-th row are filled with letters $k+1$.
The letters $k^{\prime}$ and $k$ appear only in these two rows.
The configuration of the $(i-1)$-th and $i$-th rows is
\setlength{\unitlength}{12pt}

\begin{center}
\begin{picture}(12,6)

\put(0,3){\line(0,1){1}}
\put(2,2){\line(0,1){2}}
\put(3,2){\line(0,1){1}}
\put(5,3){\line(0,1){1}}
\put(7,2){\line(0,1){2}}
\put(8,2){\line(0,1){2}}
\put(9,3){\line(0,1){1}}
\put(11,3){\line(0,1){1}}
\put(12,3){\line(0,1){1}}

\put(2,2){\line(1,0){6}}
\put(0,3){\line(1,0){12}}
\put(0,4){\line(1,0){12}}

\put(0,3){\makebox(2,1){$i-1$}}
\put(2,2){\makebox(1,1){$i$}}
\put(5,3){\makebox(2,1){$i-1$}}
\put(7,2){\makebox(1,1){$i$}}
\put(7,3){\makebox(1,1){$i^{\prime}$}}
%\put(8,2){\makebox(3,1){$(k+1)^{\prime}$}}
\put(8,3){\makebox(1,1){$i$}}
%\put(11,2){\makebox(2,1){$k+1$}}
\put(11,3){\makebox(1,1){$i$}}

\put(3,2){\makebox(4,1){$\cdots$}}
\put(2,3){\makebox(3,1){$\cdots$}}
\put(9,3){\makebox(2,1){$\cdots$}}
%\put(13,2){\makebox(2,1){$\cdots$}}

\put(3,0){\makebox(4,1){$\vdots$}}
\put(3,5){\makebox(4,1){$\vdots$}}

\end{picture},
\end{center}
where letters $i^{\prime}$ and $i$ appear only in these two rows.
The configuration below the $i$-th row of $T^{(1)}$ is the same as that of $T$.
Note that the number of $k$ in the $k$-th row is $\lambda_{i}$ irrespective of $k$ for $k=1,2,\ldots,i$.

Secondly, we compute the configuration of the tableau $S_{2}\cdots S_{i}S_{1}\cdots S_{i-1}T$.
Let us write $\mathrm{wt}(S_{i}T^{(1)})=(\nu_{1}^{(i)},\nu_{2}^{(i)},\ldots,\nu_{n}^{(i)})$.
Since $S_{i}T^{(1)}=\left(  \Tilde{f}_{i}^{P}\right)  ^{(\mu_{i}^{(1)}-\mu_{i+1}^{(1)})}T^{(1)}=
\left(  \Tilde{f}_{i}^{P}\right)^{(\lambda_{i-1}-\lambda_{i+1})}T^{(1)}$, 
we have $\nu_{i}^{(i)}=\lambda_{i-1}-(\lambda_{i-1}-\lambda_{i+1})=\lambda_{i+1}$ and 
$\nu_{i+1}^{(i)}=\lambda_{i+1}+(\lambda_{i-1}-\lambda_{i+1})=\lambda_{i-1}$.
The other components of $\mathrm{wt}(S_{i}T^{(1)})$ are the same as those of $\mathrm{wt}(T^{(1)})$.
In particular, $\nu_{i-1}^{(i)}=\lambda_{i-2}$.
In this step, the left $\lambda_{i+1}$ letters $i$ in the $i$-th row in the configuration above is bracketed in $\mathrm{rw}(T^{(1)})$ by the $(i+1,i)$-bracketing so that the number of unbracketed letters $i^{\prime}$ and $i$ is $\lambda_{i-1}-\lambda_{i+1}$.
Since $\Tilde{f}_{i}^{P}$ is applied $\lambda_{i-1}-\lambda_{i+1}$ times in this step, all of these unbracketed letters $i^{\prime}$ and $i$ are updated and the configuration of the $(i-1)$-th and $i$-th rows of $S_{i}T^{(1)}$ turns out to be
\setlength{\unitlength}{12pt}

\begin{center}
\begin{picture}(24,6)

\put(0,3){\line(0,1){1}}
\put(2,2){\line(0,1){2}}
\put(3,2){\line(0,1){1}}
\put(5,2){\line(0,1){1}}
\put(6,2){\line(0,1){1}}
\put(9,2){\line(0,1){1}}
\put(11,2){\line(0,1){1}}
\put(12,3){\line(0,1){1}}
\put(14,2){\line(0,1){2}}
\put(17,2){\line(0,1){2}}
\put(19,3){\line(0,1){1}}
\put(22,3){\line(0,1){1}}
\put(24,3){\line(0,1){1}}

\put(2,2){\line(1,0){15}}
\put(0,3){\line(1,0){24}}
\put(0,4){\line(1,0){24}}

\put(0,3){\makebox(2,1){$i-1$}}
\put(2,2){\makebox(1,1){$i$}}
\put(5,2){\makebox(1,1){$i$}}
\put(6,2){\makebox(3,1){$(i+1)^{\prime}$}}
\put(9,2){\makebox(2,1){$i+1$}}
\put(12,3){\makebox(2,1){$i-1$}}
\put(14,2){\makebox(3,1){$i+1$}}
\put(14,3){\makebox(3,1){$(i+1)^{\prime}$}}
\put(17,3){\makebox(2,1){$i+1$}}
\put(22,3){\makebox(2,1){$i+1$}}

\put(3,2){\makebox(2,1){$\cdots$}}
\put(2,3){\makebox(10,1){$\cdots$}}
\put(11,2){\makebox(3,1){$\cdots$}}
\put(19,3){\makebox(3,1){$\cdots$}}

\put(11,0){\makebox(3,1){$\vdots$}}
\put(11,5){\makebox(3,1){$\vdots$}}

\put(5,0){\makebox(1,1){\small $\lambda_{i+1}$}}
\put(5,1){\makebox(1,1){$\uparrow$}}
\end{picture}.
\end{center}
The computation can be done step by step following the rules for $\Tilde{f}_{i}^{P}$.
The letter $i^{\prime}$ disappears in the new configuration and the number of letters $i$, all of which are located in the $i$-th row, is reduced to $\lambda_{i+1}$.

Let us write $S_{k+1}\cdots S_{i}T^{(1)}=T^{(k+1)\prime}$ and 
$\mathrm{wt}(T^{(k+1)\prime})=(\nu_{1}^{(k+1)},\nu_{2}^{(k+1)},\ldots,\nu_{n}^{(k+1)})$ for $k=i-1,\ldots,2$.
For $k=i-1,\ldots,2$, assume that the configuration of the $(k-1)$-th and $k$-th rows of $T^{(k+1)\prime}$ is
\setlength{\unitlength}{12pt}

\begin{center}
\begin{picture}(16,6)

\put(0,3){\line(0,1){1}}
\put(2,2){\line(0,1){2}}
\put(3,2){\line(0,1){1}}
\put(5,3){\line(0,1){1}}
\put(7,2){\line(0,1){2}}
\put(8,2){\line(0,1){2}}
\put(11,2){\line(0,1){2}}
\put(13,2){\line(0,1){1}}
\put(15,3){\line(0,1){1}}
\put(16,3){\line(0,1){1}}

\put(2,2){\line(1,0){13}}
\put(0,3){\line(1,0){16}}
\put(0,4){\line(1,0){16}}

\put(0,3){\makebox(2,1){$k-1$}}
\put(2,2){\makebox(1,1){$k$}}
\put(5,3){\makebox(2,1){$k-1$}}
\put(7,2){\makebox(1,1){$k$}}
\put(7,3){\makebox(1,1){$k^{\prime}$}}
\put(8,2){\makebox(3,1){$(k+2)^{\prime}$}}
\put(8,3){\makebox(3,1){$k$}}
\put(11,2){\makebox(2,1){$k+2$}}
\put(15,3){\makebox(1,1){$k$}}

\put(3,2){\makebox(4,1){$\cdots$}}
\put(2,3){\makebox(3,1){$\cdots$}}
\put(13,2){\makebox(2,1){$\cdots$}}
\put(11,3){\makebox(4,1){$\cdots$}}

\put(8,0){\makebox(3,1){$\vdots$}}
\put(8,5){\makebox(3,1){$\vdots$}}

\put(7,0){\makebox(1,1){\small $\lambda_{i+1}$}}
\put(7,1){\makebox(1,1){$\uparrow$}}
\end{picture},
\end{center}
the number of letters $(k+1)$, all of which are located in the left side of the $(k+1)$-th row, is $\lambda_{i+1}$ and the letter $(k+1)^{\prime}$ does not appear in $T^{(k+1)\prime}$ so that $\nu_{k+1}^{(k+1)}=\lambda_{i+1}$, 
and letters $k^{\prime}$ and $k$ appear only in these two rows so that $\nu_{k}^{(k+1)}=\lambda_{k-1}$.
For $k=i-1$, this configuration is consistent with the previous one and other assumptions are satisfied.
Let us write 
$\mathrm{wt}(S_{k}T^{(k+1)\prime})=(\nu_{1}^{(k)},\nu_{2}^{(k)},\ldots,\nu_{n}^{(k)})$.
Since 
$T^{(k)\prime}=S_{k}T^{(k+1)\prime}=\left(  \Tilde{f}_{k}^{P}\right)  ^{(\nu_{k}^{(k+1)}-\nu_{k+1}^{(k+1)})}T^{(k+1)\prime}=
\left(  \Tilde{f}_{k}^{P}\right)^{(\lambda_{k-1}-\lambda_{i+1})}T^{(k+1)\prime}$, 
we have $\nu_{k}^{(k)}=\lambda_{k-1}-(\lambda_{k-1}-\lambda_{i+1})=\lambda_{i+1}$ and 
$\nu_{k+1}^{(k)}=\lambda_{i+1}+(\lambda_{k-1}-\lambda_{i+1})=\lambda_{k-1}$.
The other components of $\mathrm{wt}(S_{k}T^{(k+1)\prime})$ are the same as those of $\mathrm{wt}(T^{(k+1)\prime})$.
In particular, $\nu_{k-1}^{(k)}=\lambda_{k-2}$.
It is easy to see that all $\lambda_{i+1}$ letters $k+1$ are bracketed and 
the number of unbracketed letters $k^{\prime}$ and $k$ is $\lambda_{k-1}-\lambda_{i+1}$ in $\mathrm{rw}(T^{(k+1)\prime})$ when the $(k+1,k)$-bracketing is applied to $\mathrm{rw}(T^{(k+1)\prime})$. 
Since $\Tilde{f}_{k}^{P}$ is applied $\lambda_{k-1}-\lambda_{i+1}$ times in this step, all of these unbracketed letters are updated and the configuration of the $(k-2)$-th and $(k-1)$-th rows of $T^{(k)\prime}$ is computed as
\setlength{\unitlength}{12pt}

\begin{center}
\begin{picture}(20,6)

\put(0,3){\line(0,1){1}}
\put(2,2){\line(0,1){2}}
\put(4,2){\line(0,1){1}}
\put(6,3){\line(0,1){1}}
\put(8,2){\line(0,1){2}}
\put(11,2){\line(0,1){2}}
\put(14,2){\line(0,1){2}}
\put(16,2){\line(0,1){1}}
\put(18,3){\line(0,1){1}}
\put(20,3){\line(0,1){1}}

\put(2,2){\line(1,0){16}}
\put(0,3){\line(1,0){20}}
\put(0,4){\line(1,0){20}}

\put(0,3){\makebox(2,1){$k-2$}}
\put(2,2){\makebox(2,1){$k-1$}}
\put(6,3){\makebox(2,1){$k-2$}}
\put(8,2){\makebox(3,1){$k-1$}}
\put(8,3){\makebox(3,1){$(k-1)^{\prime}$}}
\put(11,2){\makebox(3,1){$(k+1)^{\prime}$}}
\put(11,3){\makebox(3,1){$k-1$}}
\put(14,2){\makebox(2,1){$k+1$}}
\put(18,3){\makebox(2,1){$k-1$}}

\put(4,2){\makebox(4,1){$\cdots$}}
\put(2,3){\makebox(4,1){$\cdots$}}
\put(16,2){\makebox(2,1){$\cdots$}}
\put(14,3){\makebox(4,1){$\cdots$}}

\put(11,0){\makebox(3,1){$\vdots$}}
\put(11,5){\makebox(3,1){$\vdots$}}

\put(8,0){\makebox(3,1){\small $\lambda_{i+1}$}}
\put(8,1){\makebox(3,1){$\uparrow$}}
\end{picture}.
\end{center}
The number of letters $k$ below the $k$-th row, all of which are located in the left side of the $k$-th row, is $\lambda_{k-1}-(\lambda_{k-1}-\lambda_{i+1})=\lambda_{i+1}$ and the letter $k^{\prime}$ disappears in the new configuration so that $\nu_{k}^{(k)}=\lambda_{i+1}$.
Hence, by induction, we have that the configuration of the first row of 
$T^{(2)\prime}=S_{2}\cdots S_{i}S_{1}\cdots S_{i-1}T=S_{w_{i}}T$ is
\setlength{\unitlength}{12pt}

\begin{center}
\begin{picture}(9,2)

\put(0,0){\line(0,1){1}}
\put(1,0){\line(0,1){1}}
\put(3,0){\line(0,1){1}}
\put(4,0){\line(0,1){1}}
\put(5,0){\line(0,1){1}}
\put(6,0){\line(0,1){1}}
\put(8,0){\line(0,1){1}}
\put(9,0){\line(0,1){1}}

\put(0,0){\line(1,0){9}}
\put(0,1){\line(1,0){9}}

\put(0,0){\makebox(1,1){$1$}}
\put(3,0){\makebox(1,1){$1$}}
\put(4,0){\makebox(1,1){$3^{\prime}$}}
\put(5,0){\makebox(1,1){$3$}}
\put(8,0){\makebox(1,1){$3$}}

\put(1,0){\makebox(2,1){$\cdots$}}
\put(6,0){\makebox(2,1){$\cdots$}}

\put(3,1){\makebox(1,1){\small $\lambda_{i}$}}
\end{picture}.
\end{center}
Since the first row contains neither $2^{\prime}$ nor $2$, $\Tilde{e}_{\Bar{1}}^{P}S_{w_{i}}T=\boldsymbol{0}$ by Lemma~\ref{lem:eP} so that 
$\Tilde{e}_{\Bar{\imath}}^{P}T=\boldsymbol{0}$ and the proof of the first part is complete. 

Let $T\in \mathrm{PT}_{n}(\lambda)$ be a primed tableau where 
the subtableau with entries $n-l(\lambda)+i$ or $(n-l(\lambda)+i)^{\prime}$ is a connected border strip of size 
$\lambda_{l(\lambda)-i+1}$ starting at $T_{i,i}=n-l(\lambda)+i$ ($i=l(\lambda).\ldots,1$).
Then, 
$\mathrm{wt}(T)=(\underbrace{0,\ldots,0}_{n-l},\lambda_{l},\ldots,\lambda_{1},0,\ldots)$, where $l=l(\lambda)$.
One of the reduced expressions of $w_{0}$ is 
\[
w_{0}=(s_{1}s_{2}\cdots s_{n-1})(s_{1}s_{2}\cdots s_{n-2})\cdots(s_{1}s_{2})(s_{1})
\]
so that 
$S_{w_{0}}T=(S_{1}S_{2}\cdots S_{n-1})\cdots(S_{1}S_{2}\cdots S_{n-l})T$.
Since $S_{n-l}T=\left(  \Tilde{e}_{n-l}^{P}\right)  ^{\lambda_{l}}T$, we have 
$\mathrm{wt}(S_{n-l}T)=(\underbrace{0,\ldots,0}_{n-l-1},\lambda_{l},0,\lambda_{l-1},\ldots,\lambda_{1},0,\ldots)$.
Repeating the similar calculations, we have 
$\mathrm{wt}(S_{1}S_{2}\cdots S_{n-l}T)=(\lambda_{l},\underbrace{0,\ldots,0}_{n-l},\lambda_{l-1},\ldots,\lambda_{1},0,\ldots)$.
Set $T^{(k)}=(S_{1}S_{2}\cdots S_{n-l+k})\cdots(S_{1}S_{2}\cdots S_{n-l})T$ and assume that
\[
\mathrm{wt}(T^{(k)})=(\underbrace{\lambda_{l-k},\lambda_{l-k+1},\ldots,\lambda_{l}}_{k+1},
\underbrace{0,\ldots,0}_{n-l},\lambda_{l-k-1},\ldots,\lambda_{1},0,\ldots)
\]
for $k=0,\ldots,l-2$.
This assumption is satisfied for $k=0$.
Since $S_{n-l+k+1}T^{(k)}=\left(  \Tilde{e}_{n-l+k+1}^{P}\right)  ^{\lambda_{l-k-1}}T^{(k)}$, we have
\[
\mathrm{wt}(S_{n-l+k+1}T^{(k)})
 =(\underbrace{\lambda_{l-k},\ldots,\lambda_{l}}_{k+1},\underbrace{0,\ldots,0}_{n-l-1},\lambda_{l-k-1},0,\lambda_{l-k-2},\ldots,\lambda_{1},0,\ldots).
\]
Repeating the similar calculations, we have 
\[
\mathrm{wt}(S_{k+2}\cdots S_{n-l+k+1}T^{(k)})
=(\underbrace{\lambda_{l-k},\ldots,\lambda_{l}}_{k+1},\lambda_{l-k-1},\underbrace{0,\ldots,0}_{n-l},
\lambda_{l-k-2},\ldots,\lambda_{1},0,\ldots).
\]
Now since 
$S_{k+1}(S_{k+2}\cdots S_{n-l+k+1}T^{(k)})=
\left( \Tilde{e}_{k+1}^{P}\right)  ^{\lambda_{l-k-1}-\lambda_{l}}S_{k+2}\cdots S_{n-l+k+1}T^{(k)}$, 
we have 
\begin{align*}
&  \mathrm{wt}(S_{k+1}\cdots S_{n-l+k+1}T^{(k)})\\
&  =(\underbrace{\lambda_{l-k},\ldots,\lambda_{l-1}}_{k},\lambda_{l-k-1},\lambda_{l},
\underbrace{0,\ldots,0}_{n-l},\lambda_{l-k-2},\ldots,\lambda_{1},0,\ldots).
\end{align*}
Repeating the similar calculations, we obtain
\begin{align*}
\mathrm{wt}(T^{(k+1)})& = \mathrm{wt}(S_{1}\cdots S_{n-l+k+1}T^{(k)})\\
&  =(\underbrace{\lambda_{l-k-1},\ldots,\lambda_{l}}_{k+2},
\underbrace{0,\ldots,0}_{n-l},\lambda_{l-k-2},\ldots,\lambda_{1},0,\ldots).
\end{align*}
Hence, by induction, we have 
$\mathrm{wt}(S_{w_{0}}T)=\mathrm{wt}(T^{(l-1)})=(\lambda_{1},\lambda_{2},\ldots,\lambda_{l},0,\ldots)$, which clearly indicates that 
all the entries in the $i$-th row of $S_{w_{0}}T$ are letters $i$ for $i=1,2,\ldots,l$ 
so that $S_{w_{0}}T$ is the $\mathfrak{q}(n)$-highest weight vector in $\mathrm{PT}_{n}(\lambda)$.
This finishes the proof of the second part.

\end{proof}

\begin{ex}
Let $n=5$ and $\lambda=(5,3,1)$.
Then the $\mathfrak{q}(n)$-highest weight vector (left) and the $\mathfrak{q}(n)$-lowest weight vector (right) are

\setlength{\unitlength}{12pt}

\begin{center}
\begin{picture}(13,3)

\put(0,2){\line(0,1){1}}
\put(1,1){\line(0,1){2}}
\put(2,0){\line(0,1){3}}
\put(3,0){\line(0,1){3}}
\put(4,1){\line(0,1){2}}
\put(5,2){\line(0,1){1}}
\put(2,0){\line(1,0){1}}
\put(1,1){\line(1,0){3}}
\put(0,2){\line(1,0){5}}
\put(0,3){\line(1,0){5}}

\put(2,0){\makebox(1,1){$3$}}
\put(1,1){\makebox(1,1){$2$}}
\put(2,1){\makebox(1,1){$2$}}
\put(3,1){\makebox(1,1){$2$}}
\put(0,2){\makebox(1,1){$1$}}
\put(1,2){\makebox(1,1){$1$}}
\put(2,2){\makebox(1,1){$1$}}
\put(3,2){\makebox(1,1){$1$}}
\put(4,2){\makebox(1,1){$1$}}

\put(8,2){\line(0,1){1}}
\put(9,1){\line(0,1){2}}
\put(10,0){\line(0,1){3}}
\put(11,0){\line(0,1){3}}
\put(12,1){\line(0,1){2}}
\put(13,2){\line(0,1){1}}
\put(10,0){\line(1,0){1}}
\put(9,1){\line(1,0){3}}
\put(8,2){\line(1,0){5}}
\put(8,3){\line(1,0){5}}

\put(10,0){\makebox(1,1){$5$}}
\put(9,1){\makebox(1,1){$4$}}
\put(10,1){\makebox(1,1){$5^{\prime}$}}
\put(11,1){\makebox(1,1){$5$}}
\put(8,2){\makebox(1,1){$3$}}
\put(9,2){\makebox(1,1){$4^{\prime}$}}
\put(10,2){\makebox(1,1){$4$}}
\put(11,2){\makebox(1,1){$5^{\prime}$}}
\put(12,2){\makebox(1,1){$5$}}

\end{picture}.
\end{center}

\end{ex}

Figure~\ref{fig:PT} is an example of $\mathfrak{q}(n)$-structure of $\mathrm{PT}_{n}(\lambda)$, where $n=3$ and $\lambda=(3,1)$.

\setlength{\unitlength}{10pt}

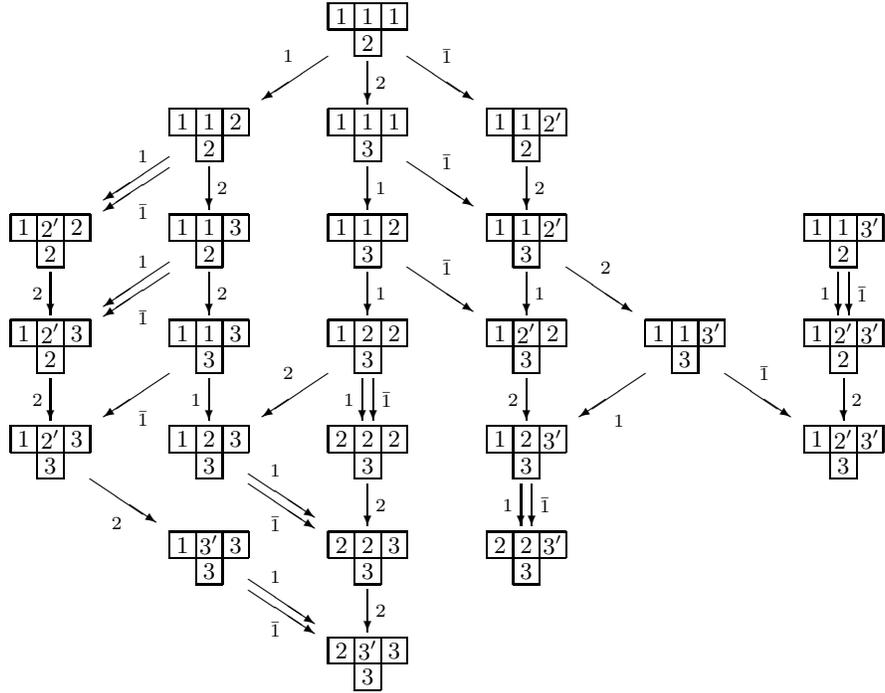
\begin{figure}
\begin{center}
\begin{picture}(33,26)

\put(1.5,11.8){\vector(0,-1){1.6}}
\put(1.5,15.8){\vector(0,-1){1.6}}

\put(7.5,11.8){\vector(0,-1){1.6}}
\put(7.5,15.8){\vector(0,-1){1.6}}
\put(7.5,19.8){\vector(0,-1){1.6}}

\put(13.5,3.8){\vector(0,-1){1.6}}
\put(13.5,7.8){\vector(0,-1){1.6}}
\put(13.3,11.8){\vector(0,-1){1.6}}
\put(13.7,11.8){\vector(0,-1){1.6}}
\put(13.5,15.8){\vector(0,-1){1.6}}
\put(13.5,19.8){\vector(0,-1){1.6}}
\put(13.5,23.8){\vector(0,-1){1.6}}

\put(19.3,7.8){\vector(0,-1){1.6}}
\put(19.7,7.8){\vector(0,-1){1.6}}
\put(19.5,11.8){\vector(0,-1){1.6}}
\put(19.5,15.8){\vector(0,-1){1.6}}
\put(19.5,19.8){\vector(0,-1){1.6}}

\put(31.5,11.8){\vector(0,-1){1.6}}
\put(31.3,15.8){\vector(0,-1){1.6}}
\put(31.7,15.8){\vector(0,-1){1.6}}

\put(3,8){\vector(3,-2){2.5}}
\put(6,12){\vector(-3,-2){2,5}}
\put(6,15.8){\vector(-3,-2){2,5}}
\put(6,16.2){\vector(-3,-2){2,5}}
\put(6,19.8){\vector(-3,-2){2,5}}
\put(6,20.2){\vector(-3,-2){2,5}}

\put(9,3.8){\vector(3,-2){2.5}}
\put(9,4.2){\vector(3,-2){2.5}}
\put(9,7.8){\vector(3,-2){2.5}}
\put(9,8.2){\vector(3,-2){2.5}}
\put(12,12){\vector(-3,-2){2,5}}
\put(12,24){\vector(-3,-2){2,5}}

\put(15,16){\vector(3,-2){2.5}}
\put(15,20){\vector(3,-2){2.5}}
\put(15,24){\vector(3,-2){2.5}}

\put(21,16){\vector(3,-2){2.5}}
\put(24,12){\vector(-3,-2){2,5}}

\put(27,12){\vector(3,-2){2.5}}

\put(0.5,10.5){\makebox(1,1){\scriptsize $2$}}
\put(0.5,14.5){\makebox(1,1){\scriptsize $2$}}

\put(6.5,10.5){\makebox(1,1){\scriptsize $1$}}
\put(7.5,14.5){\makebox(1,1){\scriptsize $2$}}
\put(7.5,18.5){\makebox(1,1){\scriptsize $2$}}

\put(13.5,2.5){\makebox(1,1){\scriptsize $2$}}
\put(13.5,6.5){\makebox(1,1){\scriptsize $2$}}
\put(12.3,10.5){\makebox(1,1){\scriptsize $1$}}
\put(13.7,10.5){\makebox(1,1){\scriptsize $\Bar{1}$}}
\put(13.5,14.5){\makebox(1,1){\scriptsize $1$}}
\put(13.5,18.5){\makebox(1,1){\scriptsize $1$}}
\put(13.5,22.5){\makebox(1,1){\scriptsize $2$}}

\put(18.3,6.5){\makebox(1,1){\scriptsize $1$}}
\put(19.7,6.5){\makebox(1,1){\scriptsize $\Bar{1}$}}
\put(18.5,10.5){\makebox(1,1){\scriptsize $2$}}
\put(19.5,14.5){\makebox(1,1){\scriptsize $1$}}
\put(19.5,18.5){\makebox(1,1){\scriptsize $2$}}

\put(31.5,10.5){\makebox(1,1){\scriptsize $2$}}
\put(30.3,14.5){\makebox(1,1){\scriptsize $1$}}
\put(31.7,14.5){\makebox(1,1){\scriptsize $\Bar{1}$}}

\put(3.5,5.8){\makebox(1,1){\scriptsize $2$}}
\put(4.5,9.8){\makebox(1,1){\scriptsize $\Bar{1}$}}
\put(4.5,13.6){\makebox(1,1){\scriptsize $\Bar{1}$}}
\put(4.5,15.7){\makebox(1,1){\scriptsize $1$}}
\put(4.5,17.6){\makebox(1,1){\scriptsize $\Bar{1}$}}
\put(4.5,19.7){\makebox(1,1){\scriptsize $1$}}

\put(9.5,1.8){\makebox(1,1){\scriptsize $\Bar{1}$}}
\put(9.5,3.8){\makebox(1,1){\scriptsize $1$}}
\put(9.5,5.8){\makebox(1,1){\scriptsize $\Bar{1}$}}
\put(9.5,7.8){\makebox(1,1){\scriptsize $1$}}
\put(10,11.5){\makebox(1,1){\scriptsize $2$}}
\put(10,23.5){\makebox(1,1){\scriptsize $1$}}

\put(16,15.5){\makebox(1,1){\scriptsize $\Bar{1}$}}
\put(16,19.5){\makebox(1,1){\scriptsize $\Bar{1}$}}
\put(16,23.5){\makebox(1,1){\scriptsize $\Bar{1}$}}

\put(22.5,9.7){\makebox(1,1){\scriptsize $1$}}
\put(22,15.5){\makebox(1,1){\scriptsize $2$}}

\put(28,11.5){\makebox(1,1){\scriptsize $\Bar{1}$}}

\put(1,8){\makebox(1,1){\small $3$}}
\put(0,9){\makebox(1,1){\small $1$}}
\put(1,9){\makebox(1,1){\small $2^{\prime}$}}
\put(2,9){\makebox(1,1){\small $3$}}

\put(1,12){\makebox(1,1){\small $2$}}
\put(0,13){\makebox(1,1){\small $1$}}
\put(1,13){\makebox(1,1){\small $2^{\prime}$}}
\put(2,13){\makebox(1,1){\small $3$}}

\put(1,16){\makebox(1,1){\small $2$}}
\put(0,17){\makebox(1,1){\small $1$}}
\put(1,17){\makebox(1,1){\small $2^{\prime}$}}
\put(2,17){\makebox(1,1){\small $2$}}

\put(7,4){\makebox(1,1){\small $3$}}
\put(6,5){\makebox(1,1){\small $1$}}
\put(7,5){\makebox(1,1){\small $3^{\prime}$}}
\put(8,5){\makebox(1,1){\small $3$}}

\put(7,8){\makebox(1,1){\small $3$}}
\put(6,9){\makebox(1,1){\small $1$}}
\put(7,9){\makebox(1,1){\small $2$}}
\put(8,9){\makebox(1,1){\small $3$}}

\put(7,12){\makebox(1,1){\small $3$}}
\put(6,13){\makebox(1,1){\small $1$}}
\put(7,13){\makebox(1,1){\small $1$}}
\put(8,13){\makebox(1,1){\small $3$}}

\put(7,16){\makebox(1,1){\small $2$}}
\put(6,17){\makebox(1,1){\small $1$}}
\put(7,17){\makebox(1,1){\small $1$}}
\put(8,17){\makebox(1,1){\small $3$}}

\put(7,20){\makebox(1,1){\small $2$}}
\put(6,21){\makebox(1,1){\small $1$}}
\put(7,21){\makebox(1,1){\small $1$}}
\put(8,21){\makebox(1,1){\small $2$}}

\put(13,0){\makebox(1,1){\small $3$}}
\put(12,1){\makebox(1,1){\small $2$}}
\put(13,1){\makebox(1,1){\small $3^{\prime}$}}
\put(14,1){\makebox(1,1){\small $3$}}

\put(13,4){\makebox(1,1){\small $3$}}
\put(12,5){\makebox(1,1){\small $2$}}
\put(13,5){\makebox(1,1){\small $2$}}
\put(14,5){\makebox(1,1){\small $3$}}

\put(13,8){\makebox(1,1){\small $3$}}
\put(12,9){\makebox(1,1){\small $2$}}
\put(13,9){\makebox(1,1){\small $2$}}
\put(14,9){\makebox(1,1){\small $2$}}

\put(13,12){\makebox(1,1){\small $3$}}
\put(12,13){\makebox(1,1){\small $1$}}
\put(13,13){\makebox(1,1){\small $2$}}
\put(14,13){\makebox(1,1){\small $2$}}

\put(13,16){\makebox(1,1){\small $3$}}
\put(12,17){\makebox(1,1){\small $1$}}
\put(13,17){\makebox(1,1){\small $1$}}
\put(14,17){\makebox(1,1){\small $2$}}

\put(13,20){\makebox(1,1){\small $3$}}
\put(12,21){\makebox(1,1){\small $1$}}
\put(13,21){\makebox(1,1){\small $1$}}
\put(14,21){\makebox(1,1){\small $1$}}

\put(13,24){\makebox(1,1){\small $2$}}
\put(12,25){\makebox(1,1){\small $1$}}
\put(13,25){\makebox(1,1){\small $1$}}
\put(14,25){\makebox(1,1){\small $1$}}

\put(19,4){\makebox(1,1){\small $3$}}
\put(18,5){\makebox(1,1){\small $2$}}
\put(19,5){\makebox(1,1){\small $2$}}
\put(20,5){\makebox(1,1){\small $3^{\prime}$}}

\put(19,8){\makebox(1,1){\small $3$}}
\put(18,9){\makebox(1,1){\small $1$}}
\put(19,9){\makebox(1,1){\small $2$}}
\put(20,9){\makebox(1,1){\small $3^{\prime}$}}

\put(19,12){\makebox(1,1){\small $3$}}
\put(18,13){\makebox(1,1){\small $1$}}
\put(19,13){\makebox(1,1){\small $2^{\prime}$}}
\put(20,13){\makebox(1,1){\small $2$}}

\put(19,16){\makebox(1,1){\small $3$}}
\put(18,17){\makebox(1,1){\small $1$}}
\put(19,17){\makebox(1,1){\small $1$}}
\put(20,17){\makebox(1,1){\small $2^{\prime}$}}

\put(19,20){\makebox(1,1){\small $2$}}
\put(18,21){\makebox(1,1){\small $1$}}
\put(19,21){\makebox(1,1){\small $1$}}
\put(20,21){\makebox(1,1){\small $2^{\prime}$}}

\put(25,12){\makebox(1,1){\small $3$}}
\put(24,13){\makebox(1,1){\small $1$}}
\put(25,13){\makebox(1,1){\small $1$}}
\put(26,13){\makebox(1,1){\small $3^{\prime}$}}

\put(31,8){\makebox(1,1){\small $3$}}
\put(30,9){\makebox(1,1){\small $1$}}
\put(31,9){\makebox(1,1){\small $2^{\prime}$}}
\put(32,9){\makebox(1,1){\small $3^{\prime}$}}

\put(31,12){\makebox(1,1){\small $2$}}
\put(30,13){\makebox(1,1){\small $1$}}
\put(31,13){\makebox(1,1){\small $2^{\prime}$}}
\put(32,13){\makebox(1,1){\small $3^{\prime}$}}

\put(31,16){\makebox(1,1){\small $2$}}
\put(30,17){\makebox(1,1){\small $1$}}
\put(31,17){\makebox(1,1){\small $1$}}
\put(32,17){\makebox(1,1){\small $3^{\prime}$}}

\put(0,9){\line(0,1){1}}
\put(1,8){\line(0,1){2}}
\put(2,8){\line(0,1){2}}
\put(3,9){\line(0,1){1}}
\put(1,8){\line(1,0){1}}
\put(0,9){\line(1,0){3}}
\put(0,10){\line(1,0){3}}

\put(0,13){\line(0,1){1}}
\put(1,12){\line(0,1){2}}
\put(2,12){\line(0,1){2}}
\put(3,13){\line(0,1){1}}
\put(1,12){\line(1,0){1}}
\put(0,13){\line(1,0){3}}
\put(0,14){\line(1,0){3}}

\put(0,17){\line(0,1){1}}
\put(1,16){\line(0,1){2}}
\put(2,16){\line(0,1){2}}
\put(3,17){\line(0,1){1}}
\put(1,16){\line(1,0){1}}
\put(0,17){\line(1,0){3}}
\put(0,18){\line(1,0){3}}

\put(6,5){\line(0,1){1}}
\put(7,4){\line(0,1){2}}
\put(8,4){\line(0,1){2}}
\put(9,5){\line(0,1){1}}
\put(7,4){\line(1,0){1}}
\put(6,5){\line(1,0){3}}
\put(6,6){\line(1,0){3}}

\put(6,9){\line(0,1){1}}
\put(7,8){\line(0,1){2}}
\put(8,8){\line(0,1){2}}
\put(9,9){\line(0,1){1}}
\put(7,8){\line(1,0){1}}
\put(6,9){\line(1,0){3}}
\put(6,10){\line(1,0){3}}

\put(6,13){\line(0,1){1}}
\put(7,12){\line(0,1){2}}
\put(8,12){\line(0,1){2}}
\put(9,13){\line(0,1){1}}
\put(7,12){\line(1,0){1}}
\put(6,13){\line(1,0){3}}
\put(6,14){\line(1,0){3}}

\put(6,17){\line(0,1){1}}
\put(7,16){\line(0,1){2}}
\put(8,16){\line(0,1){2}}
\put(9,17){\line(0,1){1}}
\put(7,16){\line(1,0){1}}
\put(6,17){\line(1,0){3}}
\put(6,18){\line(1,0){3}}

\put(6,21){\line(0,1){1}}
\put(7,20){\line(0,1){2}}
\put(8,20){\line(0,1){2}}
\put(9,21){\line(0,1){1}}
\put(7,20){\line(1,0){1}}
\put(6,21){\line(1,0){3}}
\put(6,22){\line(1,0){3}}

\put(12,1){\line(0,1){1}}
\put(13,0){\line(0,1){2}}
\put(14,0){\line(0,1){2}}
\put(15,1){\line(0,1){1}}
\put(13,0){\line(1,0){1}}
\put(12,1){\line(1,0){3}}
\put(12,2){\line(1,0){3}}

\put(12,5){\line(0,1){1}}
\put(13,4){\line(0,1){2}}
\put(14,4){\line(0,1){2}}
\put(15,5){\line(0,1){1}}
\put(13,4){\line(1,0){1}}
\put(12,5){\line(1,0){3}}
\put(12,6){\line(1,0){3}}

\put(12,9){\line(0,1){1}}
\put(13,8){\line(0,1){2}}
\put(14,8){\line(0,1){2}}
\put(15,9){\line(0,1){1}}
\put(13,8){\line(1,0){1}}
\put(12,9){\line(1,0){3}}
\put(12,10){\line(1,0){3}}

\put(12,13){\line(0,1){1}}
\put(13,12){\line(0,1){2}}
\put(14,12){\line(0,1){2}}
\put(15,13){\line(0,1){1}}
\put(13,12){\line(1,0){1}}
\put(12,13){\line(1,0){3}}
\put(12,14){\line(1,0){3}}

\put(12,17){\line(0,1){1}}
\put(13,16){\line(0,1){2}}
\put(14,16){\line(0,1){2}}
\put(15,17){\line(0,1){1}}
\put(13,16){\line(1,0){1}}
\put(12,17){\line(1,0){3}}
\put(12,18){\line(1,0){3}}

\put(12,21){\line(0,1){1}}
\put(13,20){\line(0,1){2}}
\put(14,20){\line(0,1){2}}
\put(15,21){\line(0,1){1}}
\put(13,20){\line(1,0){1}}
\put(12,21){\line(1,0){3}}
\put(12,22){\line(1,0){3}}

\put(12,25){\line(0,1){1}}
\put(13,24){\line(0,1){2}}
\put(14,24){\line(0,1){2}}
\put(15,25){\line(0,1){1}}
\put(13,24){\line(1,0){1}}
\put(12,25){\line(1,0){3}}
\put(12,26){\line(1,0){3}}

\put(18,5){\line(0,1){1}}
\put(19,4){\line(0,1){2}}
\put(20,4){\line(0,1){2}}
\put(21,5){\line(0,1){1}}
\put(19,4){\line(1,0){1}}
\put(18,5){\line(1,0){3}}
\put(18,6){\line(1,0){3}}

\put(18,9){\line(0,1){1}}
\put(19,8){\line(0,1){2}}
\put(20,8){\line(0,1){2}}
\put(21,9){\line(0,1){1}}
\put(19,8){\line(1,0){1}}
\put(18,9){\line(1,0){3}}
\put(18,10){\line(1,0){3}}

\put(18,13){\line(0,1){1}}
\put(19,12){\line(0,1){2}}
\put(20,12){\line(0,1){2}}
\put(21,13){\line(0,1){1}}
\put(19,12){\line(1,0){1}}
\put(18,13){\line(1,0){3}}
\put(18,14){\line(1,0){3}}

\put(18,17){\line(0,1){1}}
\put(19,16){\line(0,1){2}}
\put(20,16){\line(0,1){2}}
\put(21,17){\line(0,1){1}}
\put(19,16){\line(1,0){1}}
\put(18,17){\line(1,0){3}}
\put(18,18){\line(1,0){3}}

\put(18,21){\line(0,1){1}}
\put(19,20){\line(0,1){2}}
\put(20,20){\line(0,1){2}}
\put(21,21){\line(0,1){1}}
\put(19,20){\line(1,0){1}}
\put(18,21){\line(1,0){3}}
\put(18,22){\line(1,0){3}}

\put(24,13){\line(0,1){1}}
\put(25,12){\line(0,1){2}}
\put(26,12){\line(0,1){2}}
\put(27,13){\line(0,1){1}}
\put(25,12){\line(1,0){1}}
\put(24,13){\line(1,0){3}}
\put(24,14){\line(1,0){3}}

\put(30,9){\line(0,1){1}}
\put(31,8){\line(0,1){2}}
\put(32,8){\line(0,1){2}}
\put(33,9){\line(0,1){1}}
\put(31,8){\line(1,0){1}}
\put(30,9){\line(1,0){3}}
\put(30,10){\line(1,0){3}}

\put(30,13){\line(0,1){1}}
\put(31,12){\line(0,1){2}}
\put(32,12){\line(0,1){2}}
\put(33,13){\line(0,1){1}}
\put(31,12){\line(1,0){1}}
\put(30,13){\line(1,0){3}}
\put(30,14){\line(1,0){3}}

\put(30,17){\line(0,1){1}}
\put(31,16){\line(0,1){2}}
\put(32,16){\line(0,1){2}}
\put(33,17){\line(0,1){1}}
\put(31,16){\line(1,0){1}}
\put(30,17){\line(1,0){3}}
\put(30,18){\line(1,0){3}}

\end{picture}
\end{center}
\caption{$\mathfrak{q}$(3)-crystal structure of $\mathrm{PT}_{3}(\lambda =(3,1))$.}
\label{fig:PT}

\end{figure}

\section{$\mathfrak{q}(m)$-crystal structure on signed unimodal factorizations with $m$ factors of reduced words of type $B$} \label{sec:unimodal}

In this section, we give a $\mathfrak{q}(m)$-crystal structure on signed unimodal factorizations with $m$ factors of reduced words of type $B$ by translating the $\mathfrak{q}(m)$-crystal structure of primed tableaux obtained in Section~\ref{sec:primed} through the primed Kra\'{s}kiewicz insertion.
The key ingredients are the Kra\'{s}kiewicz insertion and its properties as well as the concept of unimodal words.

Let us start by defining reduced words of type $B$.
Let $W_{B}^{n}$ be the Coxeter group of type $B_{n}$ (or type $C_{n}$).
In the representation of signed permutations,  $W_{B}^{n}$ is generated by simple reflections,
\[
\begin{array}{l}
s_{0}=\Bar{1}2\cdots n, \\
s_{i}=12\cdots i-1\,i+1\,i\, i+2 \cdots n \quad (i=1,2,\ldots,n-1),
\end{array}
\]
where we use the $1$-line notation for signed permutations~\cite{Hum}, e.g.,
$3\Bar{2}4\Bar{1}$ stands for
$\begin{pmatrix}
1 & 2 & 3 & 4 \\
3 & \Bar{2} & 4 & \Bar{1}
\end{pmatrix}.
$
We identify the product of simple reflections with the sequence of their indices or the word.
For example, $s_{0}s_{1}s_{2}s_{0}s_{1}s_{3}$ may be written as the word $012013$.
For $w \in W_{B}^{n}$, the length of $w$ is defined to be 
$l(w)=\min \{l \mid \exists i_{1},\ldots,i_{l}, w=s_{i_{1}}\cdots s_{i_{l}} \}$ and 
a word $w$ of length $l(w)$ is referred to as a \emph{reduced word}.
We denote by $\left\vert \boldsymbol{a}\right\vert$ the number of letters in a word $\boldsymbol{a}$.
The set of all reduced words of the element $w \in W_{B}^{n}$ is denoted by $R(w)$.
The elements of $R(w)$ are classified by the type $B$ Coxeter-Knuth relations~\cite{Kra1}.

\begin{df}
A word $\boldsymbol{a}=a_{1}a_{2}\cdots a_{l}$ is called \emph{unimodal} if there exists $k$ such that 
$a_{1}>a_{2}>\cdots>a_{k}<\cdots<a_{l}$.
We call 
$\boldsymbol{a}_{\downarrow}=a_{1}a_{2}\ldots a_{k}$ the decreasing part of $\boldsymbol{a}$ and 
$\boldsymbol{a}_{\uparrow}=a_{k+1}a_{k+2}\ldots a_{l}$ the increasing part of $\boldsymbol{a}$.
\end{df}

\begin{df}
Let $P$ be a shifted tableau of shape $\lambda$ and let $P_{i}$ be the sequence of entries in the $i$-th row of $P$, reading from left to right.
If 
\begin{itemize}
\item[(1)]
$P_{l(\lambda)}P_{l(\lambda)-1}\cdots P_{1}$ is a reduced word of $w\in W_{B}^{n}$ and
\item[(2)] 
$P_{i}$ is a unimodal word of maximum length in $P_{i+1} P_{i}$ for $i=1,\ldots,l(\lambda)-1$,
\end{itemize} 
then $P$ is called a \emph{standard decomposition tableau} of $w$.
We denote the set of such tableaux by $\mathrm{SDT}_{w}(\lambda)$.
%The word $\pi_{P}$ is called the \emph{reading word} of $P$.
\end{df}

\begin{rem}
A standard decomposition tableau is also called a reduced unimodal tableau~\cite{HPS}.
\end{rem}

Let  $\boldsymbol{a}=a_{1}a_{2}\ldots a_{m}\in R(w)$ for $w\in W_{B}^{n}$.
We recursively construct a sequence of pairs of tableaux, 
\[
(\emptyset,\emptyset)=(P^{(0)},Q^{(0)}),(P^{(1)},Q^{(1)}),\ldots,(P^{(m)},Q^{(m)})=(P,Q),
\]
by the insertion algorithm due to Kra\'{s}kiewicz~\cite{Kra1,Kra2}.
The algorithm to obtain $(P^{(i)},Q^{(i)})$ from $(P^{(i-1)},Q^{(i-1)})$ is as follows:
\begin{itemize}
\item[(1)]
Let $a=a_{i}$ and $R$ be the first row of $P^{(i-1)}$.

\item[(2)]
Insert $a$ into $R$ as follows:

\begin{itemize}

%\item[(i):] $R=\emptyset$.
%If the empty row is the $k$-th row, we put \framebox[12pt]{\vphantom{$i$}$a$} on the position $(k+1,k+1)$ of  $P^{(i-1)}$.
%This tableau is  $P^{(i)}$.
%To get  $Q^{(i)}$, we add \framebox[12pt]{$i$} to $Q^{(i-1)}$ so that  $P^{(i)}$ and  $Q^{(i)}$ have the same shape.
%Stop.

\item[(i)]
If $Ra$ is unimodal ($R$ may be empty), append \framebox[12pt]{\vphantom{$i$}$a$} to $R$ and let  $P^{(i)}$ be this new tableau.
To get  $Q^{(i)}$, we add \framebox[12pt]{$i$} to $Q^{(i-1)}$ such that  $P^{(i)}$ and  $Q^{(i)}$ have the same shifted shape.
Stop.

\item[(ii)]
If $Ra$ is not unimodal, let $b$ be the smallest element in $R_{\uparrow}$, the increasing part of $R$, larger than or equal to $a$.
Such a $b$ always exists.

(iii-1): $a=0$ and $R$ contains $101$ as a subsequence.
Leave $R$ unchanged and go to (2) with $a=0$ and $R$ being the next row.

(iii-2): $b\neq a$.
Put $a$ in $b$'s position and let $c=b$.

(iii-3): $b=a$.
Leave $R_{\uparrow}$ unchanged and let $c=a+1$.

We insert $c$ into $R_{\downarrow}$, the decreasing part of $R$.
Let $d$ be the largest element in $R_{\downarrow}$ which is smaller than or equal to $c$.
Such an element always exists.

(iii-4): $d\neq c$.
Put $c$ in $d$'s place and let $a^{\prime}=d$.

(iii-5): $d=c$.
Leave $R_{\downarrow}$ unchanged and let $a^{\prime}=c-1$.

\end{itemize}

\item[(3)]
Repeat (2) with $a=a^{\prime}$ and $R$ being the next row.
\end{itemize}

The tableau $P$ (resp. $Q$) is called the insertion (resp. recording) tableau.

\begin{rem}
The above described Kra\'{s}kiewicz insertion is a $B_{n}$ (or $C_{n}$)-analogue of the Edelman-Greene insertion~\cite{EG} for the words of the symmetric group~\cite{Lam}. 
\end{rem}

\begin{ex} \label{ex:KR}
For a reduced word $012013\in R(3\Bar{2}4\Bar{1})$, we have

\setlength{\unitlength}{12pt}

\begin{center}
\begin{picture}(14,2)

\put(0,0){\makebox(2,2){$P=$}}

\put(2,1){\line(0,1){1}}
\put(3,0){\line(0,1){2}}
\put(4,0){\line(0,1){2}}
\put(5,0){\line(0,1){2}}
\put(6,1){\line(0,1){1}}
\put(3,0){\line(1,0){2}}
\put(2,1){\line(1,0){4}}
\put(2,2){\line(1,0){4}}

\put(2,1){\makebox(1,1){$2$}}
\put(3,0){\makebox(1,1){$0$}}
\put(3,1){\makebox(1,1){$0$}}
\put(4,0){\makebox(1,1){$1$}}
\put(4,1){\makebox(1,1){$1$}}
\put(5,1){\makebox(1,1){$3$}}

\put(8,0){\makebox(2,2){$Q=$}}

\put(10,1){\line(0,1){1}}
\put(11,0){\line(0,1){2}}
\put(12,0){\line(0,1){2}}
\put(13,0){\line(0,1){2}}
\put(14,1){\line(0,1){1}}
\put(11,0){\line(1,0){2}}
\put(10,1){\line(1,0){4}}
\put(10,2){\line(1,0){4}}

\put(10,1){\makebox(1,1){$1$}}
\put(11,0){\makebox(1,1){$4$}}
\put(11,1){\makebox(1,1){$2$}}
\put(12,0){\makebox(1,1){$5$}}
\put(12,1){\makebox(1,1){$3$}}
\put(13,1){\makebox(1,1){$6$}}

\end{picture}.
\end{center}
 
\end{ex}

\begin{thm}[\cite{Kra1,Kra2}]
The Kra\'{s}kiewicz insertion gives a bijection between $R(w)$ and pairs of tableaux $(P,Q)$ where 
$P\in \mathrm{SDT}_{w}(\lambda)$ and $Q\in \mathrm{ST}_{n}(\lambda)$ for $w\in W_{B}^{n}$.
\[
\mathrm{KR}:R(w)\stackrel{\sim}{\longrightarrow}
%TCIMACRO{\tbigcup \limits_{\lambda}}%
%BeginExpansion
{\textstyle\bigsqcup\limits_{\lambda \in \mathcal{P}^{+}}}
%EndExpansion
\left[  \mathrm{SDT}_{w}(\lambda)\times\mathrm{ST}_{n}(\lambda)\right].
\]
\end{thm}

\begin{thm}[\cite{Kra1}] \label{thm:Coxeter-Knuth}
Two reduced words in $R(w)$ are type $B$ Coxeter-Knuth related if and only if they have the same insertion tableau.
\end{thm}

\begin{df} \label{df:vee}
Suppose that the entries $i,i+1,\ldots,j$ appear in $Q$ in the following manner: 
There exists some  $i\leq k\leq j$ such that:
\begin{itemize}
\item[(1)]
the entries $i,i+1,\ldots,k$ form a vertical strip,
\item[(2)]
these entries are increasing down the vertical strip,
\item[(3)]
the entries $k,k+1,\ldots,j$ form a horizontal strip,
\item[(4)]
these entries are increasing from left to right, and 
\item[(5)]
each box in the vertical strip is left of these boxes in the horizontal strip that are on the same row. 
\end{itemize}
We say that these entries form a \emph{vee} in $Q$.
\end{df}

\begin{lem}[\cite{Lam}] \label{lem:vee}
Let $\boldsymbol{a}=a_{1}a_{2}\ldots a_{m}\in R(w)$ and $\mathrm{KR}(\boldsymbol{a})=(P,Q)$.
The subword $a_{i}a_{i+1}\ldots a_{j}$ is unimodal if and only if the entries 
$i,i+1,\ldots,j$ form a vee in $Q$.
\end{lem}

Let 
$(x_{i},y_{i}),\ldots,(x_{j},y_{j})$ be the sequence of positions of a vee $i,i+1,\ldots,j$ in $Q$.
Then, by Definition~\ref{df:vee}, there exists an index $k$ such that 
$x_{i}<\cdots<x_{k}\geq\cdots\geq x_{j}$ and $y_{i}\geq\cdots\geq y_{k}<\cdots<y_{j}$.
We call such a $k$ the \emph{bottom index}.
%In particular, if $x_{1}\geq\cdots\geq x_{r}$ and $y_{1}<\cdots<y_{r}$, then the bottom index is $1$.

\begin{ex}
In Example~\ref{ex:KR}, the entries $3,4,5,6$ in $Q$ form a vee and the corresponding subword $2013$ of $w=012013$ is unimodal.
In this example, the bottom index is $2$.
\end{ex}

For a reduced word $\boldsymbol{a}\in R(w)$, a \emph{signed unimodal factorization} of $\boldsymbol{a}$ is a factorization  $\boldsymbol{A}=(s_{1}\boldsymbol{a}_{1})(s_{2}\boldsymbol{a}_{2})\cdots$, 
where $\boldsymbol{a}_{i}$ are unimodal words and $s_{i}$ are signs.
Factors $(s_{i}\boldsymbol{a}_{i})$ can be empty.
In this case, the sign is taken to be 0.
Otherwise, signs are either  $+$ or $-$.
The weight of a unimodal factorization, denoted by $\mathrm{wt}(\boldsymbol{A})$, is defined to be the vector whose $i$-th component is $\left\vert \boldsymbol{a}_{i}\right\vert$.
We denote by $U_{m}^{\pm}(w)$ the set of all signed unimodal factorizations 
with $m$ factors of reduced words of $w\in W_{B}^{n}$.

\begin{df}
A \emph{signed primed tableau} $T$ of shape $\lambda$ is a filling of $S(\lambda)$ with letters from the alphabet $X_{n}^{\prime}$ such that: 
\begin{itemize}
\item[(1)] the entries are weakly increasing along each column and each row of $T$,
\item[(2)] each row contains at most one $i^{\prime}$ for every $i=1,\ldots,n$, and 
\item[(3)] each column contains at most one $i$ for every $i=1,\ldots,n$.
%\item[(4)] there is no primed letters on the main diagonal.
\end{itemize}
We denote by $\mathrm{PT}^{\pm}_{n}(\lambda)$ the set of signed primed tableaux of shape $\lambda$. 
\end{df}

Now, let us describe the \emph{primed Kra\'{s}kiewicz insertion} introduced in \cite{HPS}.
Let $\boldsymbol{A}\in U_{m}^{\pm}(w)$ be a signed unimodal factorization with $m$ unimodal factors  
$\boldsymbol{a}_{1},\boldsymbol{a}_{2},\ldots,\boldsymbol{a}_{m}$.
We recursively construct a sequence of pairs of tableaux, 
\[
(\emptyset,\emptyset)=(P^{(0)},T^{(0)}),(P^{(1)},T^{(1)}),\ldots,(P^{(m)},T^{(m)})=(P,T),
\]
where $P^{(i)}$ is computed by the Kra\'{s}kiewicz insertion of letters of $\boldsymbol{a}_{i}$, one by one from left to right into $P^{(i-1)}$.
The tableau $T^{(i)}$ is a signed primed tableau of shape $\lambda^{(i)}$ and computed by the following algorithm:

Let $(x_{1},y_{1}),\ldots,(x_{r},y_{r})$ be the ordered sequence of positions of boxes newly appeared on $P^{(i)}$.
By Lemma~\ref{lem:vee}, there exists an index $k$ such that 
$x_{1}<\cdots<x_{k}\geq\cdots\geq x_{r}$ and $y_{1}\geq\cdots\geq y_{k}<\cdots<y_{r}$.
The tableau $T^{(i)}$ is obtained from $T^{(i-1)}$ by adding the boxes of $(x_{1},y_{1}),\ldots,(x_{k -1},y_{k -1})$ each filled with $i^{\prime}$ and the boxes $(x_{k +1},y_{k +1}),\ldots,(x_{r},y_{r})$ each filled with $i$.
If the sign of the factor $\boldsymbol{a}_{i}$ is $-$ (resp. $+$), the box $(x_{k},y_{k})$ is filled with the letter $i^{\prime}$ (resp. $i$).

As in the Kra\'{s}kiewicz insertion, we call such a $k$ the bottom index of the sequence $(x_{1},y_{1}),\ldots,(x_{r},y_{r})$ 
and say that the sequence of entries
\[
\underbrace{i^{\prime},\ldots ,i^{\prime},}_{k-1} i^{(\prime)},
\underbrace{i,\ldots ,i}_{r-k}
\]
forms a vee in $T$, where $i^{(\prime)}$ is either $i^{\prime}$ or $i$.
We call the leftmost entry in the above sequence a \emph{head} of the vee.

\begin{ex}
For a signed unimodal factorization $\boldsymbol{A}=(+01)(-2013) \in U_{2}^{\pm}(3\Bar{2}4\Bar{1})$ of the reduced word in Example~\ref{ex:KR}, we have

\setlength{\unitlength}{12pt}

\begin{center}
\begin{picture}(14,2)

\put(0,0){\makebox(2,2){$P=$}}

\put(2,1){\line(0,1){1}}
\put(3,0){\line(0,1){2}}
\put(4,0){\line(0,1){2}}
\put(5,0){\line(0,1){2}}
\put(6,1){\line(0,1){1}}
\put(3,0){\line(1,0){2}}
\put(2,1){\line(1,0){4}}
\put(2,2){\line(1,0){4}}

\put(2,1){\makebox(1,1){$2$}}
\put(3,0){\makebox(1,1){$0$}}
\put(3,1){\makebox(1,1){$0$}}
\put(4,0){\makebox(1,1){$1$}}
\put(4,1){\makebox(1,1){$1$}}
\put(5,1){\makebox(1,1){$3$}}

\put(8,0){\makebox(2,2){$T=$}}

\put(10,1){\line(0,1){1}}
\put(11,0){\line(0,1){2}}
\put(12,0){\line(0,1){2}}
\put(13,0){\line(0,1){2}}
\put(14,1){\line(0,1){1}}
\put(11,0){\line(1,0){2}}
\put(10,1){\line(1,0){4}}
\put(10,2){\line(1,0){4}}

\put(10,1){\makebox(1,1){$1$}}
\put(11,0){\makebox(1,1){$2^{\prime}$}}
\put(11,1){\makebox(1,1){$1$}}
\put(12,0){\makebox(1,1){$2$}}
\put(12,1){\makebox(1,1){$2^{\prime}$}}
\put(13,1){\makebox(1,1){$2$}}
\end{picture}.
\end{center}
The entries $2^{\prime},2^{\prime},2,2$ whose positions are $(1,3),(2,2),(2,3),(1,4)$ in this order form a vee in $T$.
The head of the vee is the entry $2^{\prime}$ at the position $(1,3)$.
\end{ex}

\begin{thm}[\cite{HPS}]
The primed Kra\'{s}kiewicz insertion gives a bijection between $U_{m}^{\pm}(w)$ and pairs of tableaux $(P,T)$ 
where $P\in \mathrm{SDT}_{w}(\lambda)$ and $T\in \mathrm{PT}_{m}^{\pm}(\lambda)$ for $w\in W_{B}^{n}$.
\begin{equation} \label{eq:primedKR}
\mathrm{KR}^{\prime}:U_{m}^{\pm}(w)\stackrel{\sim}{\longrightarrow}
%TCIMACRO{\tbigcup \limits_{\lambda}}%
%BeginExpansion
{\textstyle\bigsqcup\limits_{\lambda \in \mathcal{P}^{+}}}
%EndExpansion
\left[  \mathrm{SDT}_{w}(\lambda)\times\mathrm{PT}_{m}^{\pm}(\lambda)\right].
\end{equation}
\end{thm}

Since $\mathrm{PT}_{m}(\lambda)$ is a $\mathfrak{q}(m)$-crystal, so is $\mathrm{PT}_{m}^{\pm}(\lambda)$.
The Kashiwara operators $\Tilde{e}_{i}^{\pm P}$ and  $\Tilde{f}_{i}^{\pm P}$ on $T\in \mathrm{PT}_{m}^{\pm}(\lambda)$ 
($i=1,\ldots,m-1,\Bar{1}$) are given by
\begin{align}
\Tilde{e}_{i}^{\pm P}T=\mathrm{PR}\circ\Tilde{e}_{i}^{P}\circ \mathrm{DPR}(T), \label{eq:esignedP} \\
\Tilde{f}_{i}^{\pm P}T=\mathrm{PR}\circ\Tilde{f}_{i}^{P}\circ \mathrm{DPR}(T) \label{eq:fsignedP},
\end{align}
where $\mathrm{DPR}$ removes primes on the main diagonal of $T$ and $\mathrm{PR}$ recovers primes at the original positions on the main diagonal.
Let us define a \emph{prime type} of $T\in \mathrm{PT}_{m}^{\pm}(\lambda)$ as the set of indexes $i$ 
such that the entry of  $(i,i)$-position is primed.
The tableaux in $\mathrm{PT}_{m}^{\pm}(\lambda)$ of the same prime type form a connected $\mathfrak{q}(m)$-crystal.
The $\mathfrak{q}(m)$-highest (resp. lowest) weight vector in $\mathrm{PT}_{m}^{\pm}(\lambda)$ is easily obtained by the $\mathfrak{q}(m)$-highest (resp. lowest) weight vector in $\mathrm{PT}_{m}(\lambda)$ and the given prime type.

\begin{ex}
Let $m=5$, $\lambda=(5,3,1)$, and the prime type be $\{1,3\}$.
Then the $\mathfrak{q}(m)$-highest weight vector (left) and the $\mathfrak{q}(m)$-lowest weight vector (right) in  $\mathrm{PT}_{m}^{\pm}(\lambda)$ are

\setlength{\unitlength}{12pt}

\begin{center}
\begin{picture}(13,3)

\put(0,2){\line(0,1){1}}
\put(1,1){\line(0,1){2}}
\put(2,0){\line(0,1){3}}
\put(3,0){\line(0,1){3}}
\put(4,1){\line(0,1){2}}
\put(5,2){\line(0,1){1}}
\put(2,0){\line(1,0){1}}
\put(1,1){\line(1,0){3}}
\put(0,2){\line(1,0){5}}
\put(0,3){\line(1,0){5}}

\put(2,0){\makebox(1,1){$3^{\prime}$}}
\put(1,1){\makebox(1,1){$2$}}
\put(2,1){\makebox(1,1){$2$}}
\put(3,1){\makebox(1,1){$2$}}
\put(0,2){\makebox(1,1){$1^{\prime}$}}
\put(1,2){\makebox(1,1){$1$}}
\put(2,2){\makebox(1,1){$1$}}
\put(3,2){\makebox(1,1){$1$}}
\put(4,2){\makebox(1,1){$1$}}

\put(8,2){\line(0,1){1}}
\put(9,1){\line(0,1){2}}
\put(10,0){\line(0,1){3}}
\put(11,0){\line(0,1){3}}
\put(12,1){\line(0,1){2}}
\put(13,2){\line(0,1){1}}
\put(10,0){\line(1,0){1}}
\put(9,1){\line(1,0){3}}
\put(8,2){\line(1,0){5}}
\put(8,3){\line(1,0){5}}

\put(10,0){\makebox(1,1){$5^{\prime}$}}
\put(9,1){\makebox(1,1){$4$}}
\put(10,1){\makebox(1,1){$5^{\prime}$}}
\put(11,1){\makebox(1,1){$5$}}
\put(8,2){\makebox(1,1){$3^{\prime}$}}
\put(9,2){\makebox(1,1){$4^{\prime}$}}
\put(10,2){\makebox(1,1){$4$}}
\put(11,2){\makebox(1,1){$5^{\prime}$}}
\put(12,2){\makebox(1,1){$5$}}

\end{picture}.
\end{center}

\end{ex}

The set $U_{m}^{\pm}(w)$ can be endowed with the structure of $\mathfrak{q}(m)$-crystal 
through that of $\mathrm{PT}_{m}^{\pm}(\lambda)$.
That is, we have
\begin{thm} \label{thm:factorization}
The set of all signed unimodal factorizations $U_{m}^{\pm}(w)$ with $m$ factors of  reduced words of $w\in W_{B}^{n}$ 
admits a $\mathfrak{q}(m)$-crystal structure through the bijection $\mathrm{KR}^{\prime}$ (Eq.~\eqref{eq:primedKR}).
\end{thm}

A set of connected elements in $U_{m}^{\pm}(w)$ is isomorphic to $\mathrm{PT}_{m}^{\pm}(\lambda)$ for some $\lambda$ and some prime type by this $\mathfrak{q}(m)$-isomorphism $\mathrm{KR}^{\prime}$ and all the elements in this set have the same insertion tableau $P\in \mathrm{SDT}_{w}(\lambda)$ so that these connected elements in $U_{m}^{\pm}(w)$ are the signed unimodal factorizations of type $B$ Coxeter-Knuth related reduced words that have the insertion tableau $P$ by Theorem~\ref{thm:Coxeter-Knuth}.
The Kashiwara operators on $U_{m}^{\pm}(w)$ are given by
\begin{align*}
\Tilde{e}_{i}^{F}&=\mathrm{KR}^{\prime -1}\circ (\mathrm{Id},\Tilde{e}_{i}^{\pm P}) \circ \mathrm{KR}^{\prime}, \\
\Tilde{f}_{i}^{F}&=\mathrm{KR}^{\prime -1}\circ(\mathrm{Id},\Tilde{f}_{i}^{\pm P}) \circ \mathrm{KR}^{\prime},
\end{align*}
for $i=1,\ldots, m-1,\Bar{1}$.

For example, we have $R(32\Bar{1})=\{0121,0212,2012\}$.
The signed unimodal factorizations (with $m$ factors) of the type $B$ Coxeter-Knuth related reduced words $0121$ and $0212$ form a connected $\mathfrak{q}(m)$-crystal for a fixed prime type.
Figure~\ref{fig:unimodal1} shows an example of such a $\mathfrak{q}(3)$-crystal structure.
The insertion tableau $P\in \mathrm{SDT}_{w}(\lambda)$ in $\mathrm{KR}^{\prime}$ (Eq.~\eqref{eq:primedKR}) is
\setlength{\unitlength}{12pt}

\begin{center}
\begin{picture}(3,2)

\put(0,1){\line(0,1){1}}
\put(1,0){\line(0,1){2}}
\put(2,0){\line(0,1){2}}
\put(3,1){\line(0,1){1}}
\put(1,0){\line(1,0){1}}
\put(0,1){\line(1,0){3}}
\put(0,2){\line(1,0){3}}

\put(0,1){\makebox(1,1){$2$}}
\put(1,0){\makebox(1,1){$0$}}
\put(1,1){\makebox(1,1){$1$}}
\put(2,1){\makebox(1,1){$2$}}

\end{picture},
\end{center}
i.e., $\lambda =(3,1)$.
The crystal structure of $\mathrm{PT}_{3}^{\pm}(\lambda)$ corresponding to Fig.~\ref{fig:unimodal1} is the same as in Fig.~\ref{fig:PT}, i.e., the prime type is $\emptyset$.
%The actions of $\Tilde{f}_{1}^{F}$ and $\Tilde{f}_{2}^{F}$are computed by Eq.~\eqref{eq:fF}.
Another choice of the prime type gives another $\mathfrak{q}(3)$-crystal structure of the signed unimodal factorizations of $0121$ and $0212$.
The unimodal factorizations of the reduced word $2012$ also form a $\mathfrak{q}(m)$-crystal.
Figure \ref{fig:unimodal2} shows an example of such a $\mathfrak{q}(3)$-crystal structure (only the higher weight part).
The insertion tableau $P\in \mathrm{SDT}_{w}(\lambda)$ in $\mathrm{KR}^{\prime}$ (Eq.~\eqref{eq:primedKR}) is
\setlength{\unitlength}{12pt}

\begin{center}
\begin{picture}(4,1)

\put(0,0){\line(0,1){1}}
\put(1,0){\line(0,1){1}}
\put(2,0){\line(0,1){1}}
\put(3,0){\line(0,1){1}}
\put(4,0){\line(0,1){1}}
\put(0,0){\line(1,0){4}}
\put(0,1){\line(1,0){4}}

\put(0,0){\makebox(1,1){$2$}}
\put(1,0){\makebox(1,1){$0$}}
\put(2,0){\makebox(1,1){$1$}}
\put(3,0){\makebox(1,1){$2$}}

\end{picture},
\end{center}
i.e., $\lambda =(4)$ and the prime type is $\emptyset$.

\setlength{\unitlength}{10pt}

\begin{figure}
\begin{center}
\begin{picture}(33,26)

\put(1.5,11.8){\vector(0,-1){1.6}}
\put(1.5,15.8){\vector(0,-1){1.6}}

\put(7.5,11.8){\vector(0,-1){1.6}}
\put(7.5,15.8){\vector(0,-1){1.6}}
\put(7.5,19.8){\vector(0,-1){1.6}}

\put(13.5,3.8){\vector(0,-1){1.6}}
\put(13.5,7.8){\vector(0,-1){1.6}}
\put(13.3,11.8){\vector(0,-1){1.6}}
\put(13.7,11.8){\vector(0,-1){1.6}}
\put(13.5,15.8){\vector(0,-1){1.6}}
\put(13.5,19.8){\vector(0,-1){1.6}}
\put(13.5,23.8){\vector(0,-1){1.6}}

\put(19.3,7.8){\vector(0,-1){1.6}}
\put(19.7,7.8){\vector(0,-1){1.6}}
\put(19.5,11.8){\vector(0,-1){1.6}}
\put(19.5,15.8){\vector(0,-1){1.6}}
\put(19.5,19.8){\vector(0,-1){1.6}}

\put(31.5,11.8){\vector(0,-1){1.6}}
\put(31.3,15.8){\vector(0,-1){1.6}}
\put(31.7,15.8){\vector(0,-1){1.6}}

\put(3,8){\vector(3,-2){2.5}}
\put(6,12){\vector(-3,-2){2,5}}
\put(6,15.8){\vector(-3,-2){2,5}}
\put(6,16.2){\vector(-3,-2){2,5}}
\put(6,19.8){\vector(-3,-2){2,5}}
\put(6,20.2){\vector(-3,-2){2,5}}

\put(9,3.8){\vector(3,-2){2.5}}
\put(9,4.2){\vector(3,-2){2.5}}
\put(9,7.8){\vector(3,-2){2.5}}
\put(9,8.2){\vector(3,-2){2.5}}
\put(12,12){\vector(-3,-2){2,5}}
\put(12,24){\vector(-3,-2){2,5}}

\put(15,16){\vector(3,-2){2.5}}
\put(15,20){\vector(3,-2){2.5}}
\put(15,24){\vector(3,-2){2.5}}

\put(21,16){\vector(3,-2){2.5}}
\put(24,12){\vector(-3,-2){2,5}}

\put(27,12){\vector(3,-2){2.5}}

\put(0.5,10.5){\makebox(1,1){\scriptsize $2$}}
\put(0.5,14.5){\makebox(1,1){\scriptsize $2$}}

\put(6.5,10.5){\makebox(1,1){\scriptsize $1$}}
\put(7.5,14.5){\makebox(1,1){\scriptsize $2$}}
\put(7.5,18.5){\makebox(1,1){\scriptsize $2$}}

\put(13.5,2.5){\makebox(1,1){\scriptsize $2$}}
\put(13.5,6.5){\makebox(1,1){\scriptsize $2$}}
\put(12.3,10.5){\makebox(1,1){\scriptsize $1$}}
\put(13.7,10.5){\makebox(1,1){\scriptsize $\Bar{1}$}}
\put(13.5,14.5){\makebox(1,1){\scriptsize $1$}}
\put(13.5,18.5){\makebox(1,1){\scriptsize $1$}}
\put(13.5,22.5){\makebox(1,1){\scriptsize $2$}}

\put(18.3,6.5){\makebox(1,1){\scriptsize $1$}}
\put(19.7,6.5){\makebox(1,1){\scriptsize $\Bar{1}$}}
\put(18.5,10.5){\makebox(1,1){\scriptsize $2$}}
\put(19.5,14.5){\makebox(1,1){\scriptsize $1$}}
\put(19.5,18.5){\makebox(1,1){\scriptsize $2$}}

\put(31.5,10.5){\makebox(1,1){\scriptsize $2$}}
\put(30.3,14.5){\makebox(1,1){\scriptsize $1$}}
\put(31.7,14.5){\makebox(1,1){\scriptsize $\Bar{1}$}}

\put(3.5,5.8){\makebox(1,1){\scriptsize $2$}}
\put(4.5,9.8){\makebox(1,1){\scriptsize $\Bar{1}$}}
\put(4.5,13.6){\makebox(1,1){\scriptsize $\Bar{1}$}}
\put(4.5,15.7){\makebox(1,1){\scriptsize $1$}}
\put(4.5,17.6){\makebox(1,1){\scriptsize $\Bar{1}$}}
\put(4.5,19.7){\makebox(1,1){\scriptsize $1$}}

\put(9.5,1.8){\makebox(1,1){\scriptsize $\Bar{1}$}}
\put(9.5,3.8){\makebox(1,1){\scriptsize $1$}}
\put(9.5,5.8){\makebox(1,1){\scriptsize $\Bar{1}$}}
\put(9.5,7.8){\makebox(1,1){\scriptsize $1$}}
\put(10,11.5){\makebox(1,1){\scriptsize $2$}}
\put(10,23.5){\makebox(1,1){\scriptsize $1$}}

\put(16,15.5){\makebox(1,1){\scriptsize $\Bar{1}$}}
\put(16,19.5){\makebox(1,1){\scriptsize $\Bar{1}$}}
\put(16,23.5){\makebox(1,1){\scriptsize $\Bar{1}$}}

\put(22.5,9.7){\makebox(1,1){\scriptsize $1$}}
\put(22,15.5){\makebox(1,1){\scriptsize $2$}}

\put(28,11.5){\makebox(1,1){\scriptsize $\Bar{1}$}}

\put(0,8){\makebox(3,2){\footnotesize $(+0)(-2)(+12)$}}
\put(0,12){\makebox(3,2){\footnotesize $(+0)(+21)(+2)$}}
\put(0,16){\makebox(3,2){\footnotesize $(+0)(+212)()$}}

\put(6,4){\makebox(3,2){\footnotesize $(+0)()(+212)$}}
\put(6,8){\makebox(3,2){\footnotesize $(+0)(+2)(+12)$}}
\put(6,12){\makebox(3,2){\footnotesize $(+02)()(+12)$}}
\put(6,16){\makebox(3,2){\footnotesize $(+02)(+1)(+2)$}}
\put(6,20){\makebox(3,2){\footnotesize $(+02)(+12)()$}}

\put(12,0){\makebox(3,2){\footnotesize $()(+0)(+212)$}}
\put(12,4){\makebox(3,2){\footnotesize $()(+02)(+12)$}}
\put(12,8){\makebox(3,2){\footnotesize $()(+012)(+1)$}}
\put(12,12){\makebox(3,2){\footnotesize $(+0)(+12)(+1)$}}
\put(12,16){\makebox(3,2){\footnotesize $(+01)(+2)(+1)$}}
\put(12,20){\makebox(3,2){\footnotesize $(+012)()(+1)$}}
\put(12,24){\makebox(3,2){\footnotesize $(+012)(+1)()$}}

\put(18,4){\makebox(3,2){\footnotesize $()(+01)(+21)$}}
\put(18,8){\makebox(3,2){\footnotesize $(+0)(+1)(+21)$}}
\put(18,12){\makebox(3,2){\footnotesize $(+0)(-12)(+1)$}}
\put(18,16){\makebox(3,2){\footnotesize $(+01)(-2)(+1)$}}
\put(18,20){\makebox(3,2){\footnotesize $(+01)(+21)()$}}

\put(24,12){\makebox(3,2){\footnotesize $(+01)()(+21)$}}

\put(30,8){\makebox(3,2){\footnotesize $(+0)(-1)(+21)$}}
\put(30,12){\makebox(3,2){\footnotesize $(+0)(+21)(-2)$}}
\put(30,16){\makebox(3,2){\footnotesize $(+02)(+1)(-2)$}}

\end{picture}
\end{center}
\caption{An example of $\mathfrak{q}(3)$-crystal structure of signed unimodal factorizations of the type $B$ Coxeter-Knuth related reduced words $0121$ and $0212$.}
 \label{fig:unimodal1}
\end{figure}
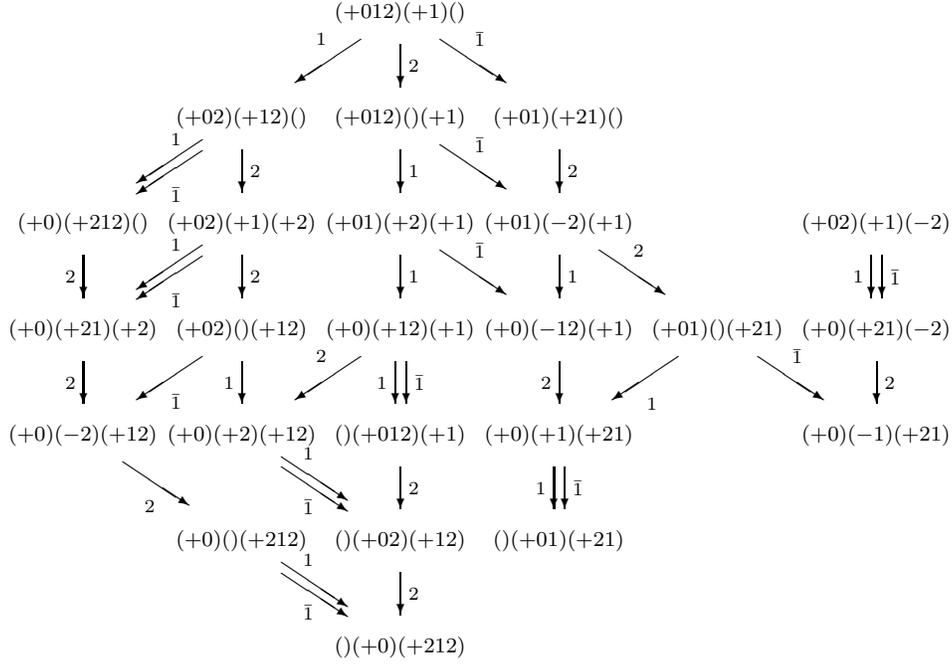

\setlength{\unitlength}{10pt}

\begin{figure}
\begin{center}
\begin{picture}(26.5,12)

\put(18.5,5.8){\vector(0,-1){1.6}}

\put(5,6){\vector(-3,-2){2,5}}
\put(11,10){\vector(-3,-2){2,5}}
\put(17,6){\vector(-3,-2){2,5}}

\put(8,6){\vector(3,-2){2.5}}
\put(14,10){\vector(3,-2){2.5}}
\put(20,6){\vector(3,-2){2,5}}

\put(18.5,4.5){\makebox(1,1){\scriptsize $1$}}

\put(3,5.5){\makebox(1,1){\scriptsize $2$}}

\put(9,5.5){\makebox(1,1){\scriptsize $1$}}
\put(9,9.5){\makebox(1,1){\scriptsize $\Bar{1}$}}

\put(15,5.5){\makebox(1,1){\scriptsize $\Bar{1}$}}
\put(15,9.5){\makebox(1,1){\scriptsize $1$}}

\put(21,5.5){\makebox(1,1){\scriptsize $2$}}

\put(0,2){\makebox(3,2){\footnotesize $(+201)()(-2)$}}

\put(5,6){\makebox(3,2){\footnotesize $(+201)(-2)()$}}

\put(10.5,2){\makebox(3,2){\footnotesize $(+20)(-12)()$}}
\put(11,10){\makebox(3,2){\footnotesize $(+2012)()()$}}

\put(17,2){\makebox(3,2){\footnotesize $(+20)(+12)()$}}
\put(17,6){\makebox(3,2){\footnotesize $(+201)(+2)()$}}

\put(23.5,2){\makebox(3,2){\footnotesize $(+201)()(+2)$}}

\put(1,0){\makebox(1,1){\vdots}}
\put(12,0){\makebox(1,1){\vdots}}
\put(18,0){\makebox(1,1){\vdots}}
\put(24,0){\makebox(1,1){\vdots}}

\end{picture}
\end{center}
\caption{An example of $\mathfrak{q}$(3)-crystal structure of signed unimodal factorizations of the reduced word $2012$.
The total number of vertices is 33.}
 \label{fig:unimodal2}
\end{figure}
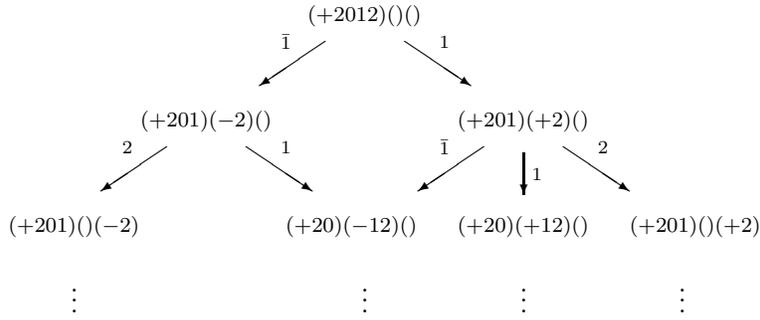

As for the action of odd Kashiwara operators 
$\Tilde{e}_{\Bar{1}}^{F}$ and $\Tilde{f}_{\Bar{1}}^{F}$ on $\boldsymbol{A}$, 
we can construct the explicit algorithms as claimed by the following two theorems (Theorem~\ref{thm:eF} and \ref{thm:fF}).

\begin{thm} \label{thm:eF}
Let $\boldsymbol{A}=(s_{1}\boldsymbol{a}_{1})(s_{2}\boldsymbol{a}_{2})\cdots\in U_{m}^{\pm}(w)$ be a signed unimodal factorization with the unimodal word $\boldsymbol{a}_{i}$ and its sign ($i=1,2,\ldots,m$) for $w\in W_{B}^{n}$.
The action of the odd Kashiwara operator $\Tilde{e}_{\Bar{1}}^{F}$ on $\boldsymbol{A}$ is given by the following rule:

$\Tilde{e}_{\Bar{1}}^{F}$ always changes the first two factors if 
$\Tilde{e}_{\Bar{1}}^{F}\boldsymbol{A}\neq \boldsymbol{0}$.
\begin{itemize}
\item[(1):]
$\boldsymbol{a}_{1}\boldsymbol{a}_{2}$ is not unimodal.
In this case,
\[
\Tilde{e}_{\Bar{1}}^{F}\boldsymbol{A}=(s_{1}\boldsymbol{a}_{1}^{\prime})(s_{2}\boldsymbol{a}_{2}^{\prime})\cdots,
\]
where $\boldsymbol{a}_{2}^{\prime}$ is obtained by dropping the first letter of $\boldsymbol{a}_{2}$ and 
$\boldsymbol{a}_{1}^{\prime}$ is the word obtained by appending that letter to $\boldsymbol{a}_{1}$.
Here, $\boldsymbol{a}_{1}^{\prime}$ is assumed to be unimodal.
Otherwise, $\Tilde{e}_{\Bar{1}}^{F}\boldsymbol{A}=\boldsymbol{0}$.

\item[(2):]
$\boldsymbol{a}_{1}\boldsymbol{a}_{2}$ is unimodal.
In this case, 
If $\boldsymbol{a}_{2}=\emptyset$ or $s_{1}\neq0$ and $s_{2}=+$, 
then $\Tilde{e}_{\Bar{1}}^{F}\boldsymbol{A}=\boldsymbol{0}$.
Otherwise,
\[
\Tilde{e}_{\Bar{1}}^{F}\boldsymbol{A}=
(s_{1}^{\prime}\boldsymbol{a}_{1}^{\prime})(s_{2}^{\prime}\boldsymbol{a}_{2}^{\prime})\cdots,
\]
where $\boldsymbol{a}_{1}^{\prime}$ and $\boldsymbol{a}_{2}^{\prime}$ are the same as in (1) 
and signs $s_{1}^{\prime}$ and $s_{2}^{\prime}$ are given by the following rule:

When $\boldsymbol{a}_{1}=\emptyset$, 
\begin{equation*}
(s_{1}^{\prime},s_{2}^{\prime})=
\begin{cases}
(s_{2},0), & \left\vert \boldsymbol{a}_{2}\right\vert =1, \\
(s_{2},+), & \left\vert \boldsymbol{a}_{2}\right\vert \geq2.
\end{cases}
\end{equation*}

When $\boldsymbol{a}_{1}\neq\emptyset$,
\begin{equation*}
(s_{1}^{\prime},s_{2}^{\prime})=
\begin{cases}
(s_{1},0), & \left\vert \boldsymbol{a}_{2}\right\vert =1, \\
(s_{1},+), & \left\vert \boldsymbol{a}_{2}\right\vert \geq2.
\end{cases}
\end{equation*}
\end{itemize}

\end{thm}

\begin{thm} \label{thm:fF}
Let $\boldsymbol{A}=(s_{1}\boldsymbol{a}_{1})(s_{2}\boldsymbol{a}_{2})\cdots\in U_{m}^{\pm}(w)$ be a signed unimodal factorization with the unimodal word $\boldsymbol{a}_{i}$ and its sign ($i=1,2,\ldots,m$) for $w\in W_{B}^{n}$.
The action of the odd Kashiwara operator $\Tilde{f}_{\Bar{1}}^{F}$ on $\boldsymbol{A}$ is given by the following rule:

$\Tilde{f}_{\Bar{1}}^{F}$ always change the first two factors if 
$\Tilde{f}_{\Bar{1}}^{F}\boldsymbol{A}\neq \boldsymbol{0}$
.\begin{itemize}
\item[(1):]
$\boldsymbol{a}_{1}\boldsymbol{a}_{2}$ is not unimodal.
\[
\Tilde{f}_{\Bar{1}}^{F}\boldsymbol{A}=(s_{1}\boldsymbol{a}%
_{1}^{\prime})(s_{2}\boldsymbol{a}_{2}^{\prime})\cdots,
\]
where $\boldsymbol{a}_{1}^{\prime}$ is obtained by dropping the last letter of $\boldsymbol{a}_{1}$ and 
$\boldsymbol{a}_{2}^{\prime}$ is the word obtained by prepending that letter to $\boldsymbol{a}_{2}$.
Here, $\boldsymbol{a}_{2}^{\prime}$ is assumed to be unimodal.
Otherwise, $\Tilde{f}_{\Bar{1}}^{F}\boldsymbol{A}=\boldsymbol{0}$.

\item[(2):]
$\boldsymbol{a}_{1}\boldsymbol{a}_{2}$ is unimodal.
If $\boldsymbol{a}_{2}\neq\emptyset$ and $s_{2}=-$, 
then $\Tilde{f}_{\Bar{1}}^{F}\boldsymbol{A}=\boldsymbol{0}$.
Otherwise,
\[
\Tilde{f}_{\Bar{1}}^{F}\boldsymbol{A}=
(s_{1}^{\prime}\boldsymbol{a}_{1}^{\prime})(s_{2}^{\prime}\boldsymbol{a}_{2}^{\prime})\cdots,
\]
where $\boldsymbol{a}_{1}^{\prime}$ and $\boldsymbol{a}_{2}^{\prime}$ are the same as in (1) 
and signs $s_{1}^{\prime}$ and $s_{2}^{\prime}$ are given by the following rule:
\begin{equation*}
(s_{1}^{\prime},s_{2}^{\prime})=
\begin{cases}
(0,s_{1}), & \left\vert \boldsymbol{a}_{1}\right\vert =1, \\
(s_{1},-), & \left\vert \boldsymbol{a}_{1}\right\vert \geq2.
\end{cases}
\end{equation*}

\end{itemize}

\end{thm}

We give the proof of Theorem~\ref{thm:eF} only.
The proof of Theorem~\ref{thm:fF} is similar.

\begin{proof}[Proof of Theorem~\ref{thm:eF}]

Let $T\in \mathrm{PT}_{m}^{\pm}(\lambda)$ be a signed primed tableau such that  $\mathrm{KR}^{\prime}(\boldsymbol{A})=(P,T)$, 
where $P \in \mathrm{SDT}_{w}(\lambda)$.
Let us write $\boldsymbol{a}_{1}=u_{1}^{(1)}\cdots u_{r}^{(1)}$ and $\boldsymbol{a}_{2}=u_{1}^{(2)}\cdots u_{s}^{(2)}$,
Suppose that $T^{\prime}=\Tilde{e}_{\Bar{1}}^{\pm P}T\neq\boldsymbol{0}$.
Namely, we have the following three cases:
\begin{itemize}
\item[(i)]
$T_{1,1}=2^{\prime}$, $T_{1,1}^{\prime}=1^{\prime}$, and $T_{i,j}=T_{i,j}^{\prime}$ for $(i,j)\neq (1,1)$.
\item[(ii)]
$T_{1,1}=2$, $T_{1,1}^{\prime}=1$, and $T_{i,j}=T_{i,j}^{\prime}$ for $(i,j)\neq (1,1)$.
\item[(iii)]
$T_{1,k}=2^{\prime}$, $T_{1,k}^{\prime}=1$ ($\exists ! k\geq2$), and $T_{i,j}=T_{i,j}^{\prime}$ for $(i,j)\neq (1,k)$.
In this case, the entry $2^{\prime}$ does not appear in the first row of $T$ except at $(1,k)$.
\end{itemize}
Let $\boldsymbol{A}^{\prime}=(s_{1}\boldsymbol{a}_{1}^{\prime})(s_{2}\boldsymbol{a}_{2}^{\prime})\cdots\in U_{m}^{\pm}(w)$ be 
a signed unimodal factorization such that $\mathrm{KR}^{\prime}(\boldsymbol{A}^{\prime})=(P,T^{\prime})$.

In Case (i), since neither $1^{\prime}$ nor $1$ appear in $T$, the first factor of $\boldsymbol{A}$ is empty; $(s_{1}\boldsymbol{a}_{1})=()$.
The entry $T_{1,1}=2^{\prime}$ appears first when the letter $u_{1}^{(2)}$ is inserted and 
neither $1^{\prime}$ nor $1$ appear at the position except at $(1,1)$ in $T^{\prime}$ so that 
$\boldsymbol{a}_{1}^{\prime}=u_{1}^{(2)}$ and $\boldsymbol{a}_{2}^{\prime}=u_{2}^{(2)}\ldots u_{s}^{(2)}$.
Only the first two factors of $\boldsymbol{A}$ change under the action of 
$\Tilde{e}_{\Bar{1}}^{F}$ 
if $\Tilde{e}_{\Bar{1}}^{F}\boldsymbol{A}\neq \boldsymbol{0}$.

In Case (ii), we have that 
$\boldsymbol{a}_{1}^{\prime}=u_{1}^{(2)}$ and $\boldsymbol{a}_{2}^{\prime}=u_{2}^{(2)}\cdots u_{s}^{(2)}$ and that 
only the first two factors of $\boldsymbol{A}$ change under the action of 
$\Tilde{e}_{\Bar{1}}^{F}$ 
if $\Tilde{e}_{\Bar{1}}^{F}\boldsymbol{A}\neq \boldsymbol{0}$ as in Case (i).

In Case (iii), the primed entry $T_{1,k}=2^{\prime}$ is the head of the vee starting from $2^{\prime}$ in $T$.
Otherwise, there exists an entry $2^{\prime}$ in $T$ at the position $(x_{1},y_{1})\neq (1,k)$ with $x_{1}<1$ and $y_{1}\geq k$, which is impossible.
Since $T_{1,k}=2^{\prime}$ is changed to $T_{1,k}^{\prime}=1$, we have that 
$\boldsymbol{a}_{1}^{\prime}=\boldsymbol{a}_{1}u_{1}^{(2)}$ and $\boldsymbol{a}_{2}^{\prime}=u_{2}^{(2)}\cdots u_{s}^{(2)}$ and that 
only the first two factors of $\boldsymbol{A}$ change under the action of 
$\Tilde{e}_{\Bar{1}}^{F}$ 
if $\Tilde{e}_{\Bar{1}}^{F}\boldsymbol{A}\neq \boldsymbol{0}$.

We first claim that $\Tilde{e}_{\Bar{1}}^{F}\boldsymbol{A}= \boldsymbol{0}$ if 
$\boldsymbol{a}_{1}^{\prime}$ is not unimodal.
Note that $\boldsymbol{a}_{2}^{\prime}$ is always unimodal.
This is verified as follows.
Suppose that $\boldsymbol{a}_{1}^{\prime}$ is not unimodal.
Then  $\boldsymbol{a}_{1}$ consists of at least two letters and $T$ has the following configuration.
(Note that $\boldsymbol{a}_{1}$ is unimodal).

\setlength{\unitlength}{12pt}

\begin{center}
\begin{picture}(7,4)

\put(0,2){\line(0,1){1}}
\put(1,1){\line(0,1){2}}
\put(2,1){\line(0,1){2}}
\put(4,2){\line(0,1){1}}
\put(5,2){\line(0,1){1}}
\put(1,1){\line(1,0){1}}
\put(0,2){\line(1,0){7}}
\put(0,3){\line(1,0){7}}

\put(0,2){\makebox(1,1){$x$}}
\put(1,1){\makebox(1,1){$y$}}
\put(1,2){\makebox(1,1){$1$}}
\put(2,2){\makebox(2,1){$\cdots$}}
\put(4,2){\makebox(1,1){$1$}}
\put(5,2){\makebox(2,1){$\cdots$}}

\put(2,0){\makebox(1,1){$\ddots$}}
\put(4,3){\makebox(1,1){\small $r$}}

\end{picture},
\end{center}
where $r\geq 2$, the entry $x$ is either $1^{\prime}$ or $1$, and the entry $y$, which appears when the letter $u_{1}^{(2)}$ is inserted, is either $2^{\prime}$ or $2$.
The entry $y$ does not appear in the first row; otherwise $\boldsymbol{a}_{1}^{\prime}$ would be unimodal.
When $\left\vert \boldsymbol{a}_{2}\right\vert =1$, the entry $2^{\prime}$ does not appear in the first row of $T$ 
so that $\Tilde{e}_{\Bar{1}}^{\pm P}T=\boldsymbol{0}$ by Lemma~\ref{lem:eP} with Eq.~\eqref{eq:esignedP} and therefore $\Tilde{e}_{\Bar{1}}^{F}\boldsymbol{A}=\boldsymbol{0}$.
When $\left\vert \boldsymbol{a}_{2}\right\vert \geq2$, suppose that the entry $z$ ($z$ is either $2^{\prime}$ or $2$) appears in the first row of $T$:

\setlength{\unitlength}{12pt}

\begin{center}
\begin{picture}(8,4)

\put(0,2){\line(0,1){1}}
\put(1,1){\line(0,1){2}}
\put(2,1){\line(0,1){2}}
\put(4,2){\line(0,1){1}}
\put(5,2){\line(0,1){1}}
\put(6,2){\line(0,1){1}}
\put(1,1){\line(1,0){1}}
\put(0,2){\line(1,0){8}}
\put(0,3){\line(1,0){8}}

\put(0,2){\makebox(1,1){$x$}}
\put(1,1){\makebox(1,1){$y$}}
\put(1,2){\makebox(1,1){$1$}}
\put(2,2){\makebox(2,1){$\cdots$}}
\put(4,2){\makebox(1,1){$1$}}
\put(5,2){\makebox(1,1){$z$}}
\put(6,2){\makebox(2,1){$\cdots$}}

\put(2,0){\makebox(1,1){$\ddots$}}
\put(4,3){\makebox(1,1){\small $r$}}

\end{picture}.
\end{center}
The entry $z$ appears after $y$ appears in the course of primed Kra\'{s}kiewicz insertion of $\boldsymbol{a}_{2}$.
The position of $y$ is $(2,2)$ and that of $z$ is $(1,r+1)$.
Let $k^{(2)}$ be the bottom index of the sequence of boxes which appear in the insertion of $\boldsymbol{a}_{2}$.
Since $2\geq 1$ and $2<r+1$, $k^{(2)}=1$ so that $z=2$ (not $2^{\prime}$).
Therefore, $\Tilde{e}_{\Bar{1}}^{\pm P}T=\boldsymbol{0}$ so that 
$\Tilde{e}_{\Bar{1}}^{F}\boldsymbol{A}=\boldsymbol{0}$.

Next, we verify the rule of signs.
We assume that $\boldsymbol{a}_{1}^{\prime}$ is unimodal.
Let 
\[
(x_{1}^{(1)},y_{1}^{(1)}),\ldots,(x_{r}^{(1)},y_{r}^{(1)})
\] 
be the ordered sequence of the positions of boxes which appear in the insertion of the unimodal word $\boldsymbol{a}_{1}$ with the bottom index $k^{(1)}$.
Similarly, let 
\[
(x_{1}^{(2)},y_{1}^{(2)}),\ldots,(x_{s}^{(2)},y_{s}^{(2)})
\]
be the ordered sequence of the positions of boxes which appear in the insertion of the unimodal word $\boldsymbol{a}_{2}$ with the bottom index $k^{(2)}$.
Let $k^{(1)\prime}$ (resp. $k^{(2)\prime}$) be the bottom index of the sequence of positions of boxes which appear in the insertion of the unimodal word $\boldsymbol{a}_{1}^{\prime}$ (resp. $\boldsymbol{a}_{2}^{\prime}$).

(1): $\boldsymbol{a}_{1}\boldsymbol{a}_{2}$ is not unimodal.
In this case, $\boldsymbol{a}_{1}\neq\emptyset$ and $\boldsymbol{a}_{2}\neq\emptyset$.
Since $\boldsymbol{a}_{1}^{\prime}$ is unimodal, we have 
$1=x_{1}^{(1)}=\cdots=x_{r}^{(1)}=x_{1}^{(2)}$ and 
$y_{1}^{(1)}<\cdots<y_{r}^{(1)}<y_{1}^{(2)}$ 
so that $k^{(1)}=1$ and $k^{(1)\prime}=1$.
If the bottom index $k^{(2)}=1$, then $1=x_{1}^{(2)}=\cdots=x_{s}^{(2)}$ and $y_{1}^{(2)}<\cdots<y_{s}^{(2)}$ so that
\[
1=x_{1}^{(1)}=\cdots=x_{r}^{(1)}=x_{1}^{(2)}=\cdots =x_{s}^{(2)}
\]
and
\[
y_{1}^{(1)}<\cdots<y_{r}^{(1)}<y_{1}^{(2)}<\cdots<y_{s}^{(2)}.
\]
This implies that the sequence of entries 
$\underbrace{1^{(\prime)},\ldots ,1,}_{r} \underbrace{2^{(\prime)},\ldots ,2}_{s}$ forms a vee in $T$, 
where $i^{(\prime)}$ is either $i^{\prime}$ or $i$ ($i=1,2$).
Consequently, 
$\boldsymbol{a}_{1}\boldsymbol{a}_{2}$ is unimodal, which contradicts the assumption of (1).
Hence, $k^{(2)}\geq2$.
Therefore, the entry at the position $(x_{1}^{(2)},y_{1}^{(2)})=(1,y_{1}^{(2)})$ in $T$ is $2^{\prime}$, 
which is head of the vee starting from $2^{\prime}$ in $T$ and is the only $2^{\prime}$ in the first row of $T$.
The operator $\Tilde{e}_{\Bar{1}}^{\pm P}$ changes this $2^{\prime}$ to $1$ and leaves other entries unchanged. 
Here, $s_{1}^{\prime}=s_{1}$ and $s_{2}^{\prime}=s_{2}$.
This can be shown as follows.
Since $\boldsymbol{a}_{1}^{\prime}$ is unimodal, the configuration of the first row of $T$ is

\setlength{\unitlength}{12pt}

\begin{center}
\begin{picture}(8,2)

\put(0,0){\line(0,1){1}}
\put(1,0){\line(0,1){1}}
\put(2,0){\line(0,1){1}}
\put(4,0){\line(0,1){1}}
\put(5,0){\line(0,1){1}}
\put(6,0){\line(0,1){1}}
\put(0,0){\line(1,0){8}}
\put(0,1){\line(1,0){8}}

\put(0,0){\makebox(1,1){$x$}}
\put(1,0){\makebox(1,1){$1$}}
\put(2,0){\makebox(2,1){$\cdots$}}
\put(4,0){\makebox(1,1){$1$}}
\put(5,0){\makebox(1,1){$2^{\prime}$}}
\put(6,0){\makebox(2,1){$\cdots$}}

\put(5,1){\makebox(1,1){\small $r+1$}}
\end{picture},
\end{center}
where $r\geq 1$ and $x$ is either $1^{\prime}$ or $1$ so that the first row of $T^{\prime}$ has the following configuration.
\setlength{\unitlength}{12pt}

\begin{center}
\begin{picture}(8,2)

\put(0,0){\line(0,1){1}}
\put(1,0){\line(0,1){1}}
\put(2,0){\line(0,1){1}}
\put(4,0){\line(0,1){1}}
\put(5,0){\line(0,1){1}}
\put(6,0){\line(0,1){1}}
\put(0,0){\line(1,0){8}}
\put(0,1){\line(1,0){8}}

\put(0,0){\makebox(1,1){$x$}}
\put(1,0){\makebox(1,1){$1$}}
\put(2,0){\makebox(2,1){$\cdots$}}
\put(4,0){\makebox(1,1){$1$}}
\put(5,0){\makebox(1,1){$1$}}
\put(6,0){\makebox(2,1){$\cdots$}}

\put(5,1){\makebox(1,1){\small $r+1$}}
\end{picture}.
\end{center}
Then, we have that $s_{1}^{\prime}=s_{1}$; otherwise $T_{1,1}^{\prime}\neq x$ because $k^{(1)}=1$ and $k^{(1)\prime}=1$.
By noting $k^{(2)}\geq2$, we also have that $s_{2}^{\prime}=s_{2}$ because the entries $2^{\prime}$ and $2$ at the same positions except at $(1,r+1)$ in $T$ and $T^{\prime}$ coincide.

(2):$\boldsymbol{a}_{1}\boldsymbol{a}_{2}$ is unimodal.
In this case, 
\[
1=x_{1}^{(1)}=\cdots=x_{r}^{(1)}=x_{1}^{(2)}=\cdots=x_{s}^{(2)}
\]
and
\[
y_{1}^{(1)}<\cdots<y_{r}^{(1)}<y_{1}^{(2)}\cdots<y_{s}^{(2)}
\]
so that $k^{(2)}=1$.
If $\boldsymbol{a}_{1}\neq\emptyset$ and $s_{2}=+$, then the configuration of the first row of $T$ is
\setlength{\unitlength}{12pt}

\begin{center}
\begin{picture}(7,2)

\put(0,0){\line(0,1){1}}
\put(1,0){\line(0,1){1}}
\put(3,0){\line(0,1){1}}
\put(4,0){\line(0,1){1}}
\put(5,0){\line(0,1){1}}
\put(0,0){\line(1,0){7}}
\put(0,1){\line(1,0){7}}

\put(0,0){\makebox(1,1){$x$}}
\put(1,0){\makebox(2,1){$\cdots$}}
\put(3,0){\makebox(1,1){$1$}}
\put(4,0){\makebox(1,1){$2$}}
\put(5,0){\makebox(2,1){$\cdots$}}

\put(3,1){\makebox(1,1){\small $r$}}
\end{picture},
\end{center}
where $r\geq 1$ and $x$ is either $1^{\prime}$ or $1$.
The entry $2$ in the above configuration is unprimed because $k^{(2)}=1$ and $s_{2}=+$.
Hence, $\Tilde{e}_{\Bar{1}}^{\pm P}T=\boldsymbol{0}$ so that 
$\Tilde{e}_{\Bar{1}}^{F}\boldsymbol{A}=\boldsymbol{0}$.
In the following, we assume that $\boldsymbol{a}_{1}=\emptyset$ or $s_{2}\neq+$.

\begin{itemize}
\item[(i):]
$\boldsymbol{a}_{1}=\emptyset$.

\begin{itemize}
\item[(i-1):]
$\left\vert \boldsymbol{a}_{2}\right\vert =1$.
In this case, $T_{1,1}$ is either $2^{\prime}$ or $2$.
If $T_{1,1}=2^{\prime}$, then $s_{2}=-$ and $T_{1,1}^{\prime}=1^{\prime}$, i.e., $s_{1}^{\prime}=-$.
On the other hand, if $T_{1,1}=2$, then $s_{2}=+$ and $T_{1,1}^{\prime}=1$, i.e., $s_{1}^{\prime}=+$.
In both cases, $s_{1}^{\prime}=s_{2}$.

\item[(i-2):]
$\left\vert \boldsymbol{a}_{2}\right\vert \geq2$.
As in (i-1), we have $s_{1}^{\prime}=s_{2}$.
Since
$1=x_{1}^{(2)}=\cdots=x_{s}^{(2)}$ and $y_{1}^{(2)}<\cdots<y_{s}^{(2)}$,
$k^{(2)}=1$.
We claim that $s_{2}^{\prime}=+$.
In the case when $s_{2}=+$, the configuration of the first row of $T$ is
\setlength{\unitlength}{12pt}

\begin{center}
\begin{picture}(4,1)

\put(0,0){\line(0,1){1}}
\put(1,0){\line(0,1){1}}
\put(2,0){\line(0,1){1}}
\put(0,0){\line(1,0){4}}
\put(0,1){\line(1,0){4}}

\put(0,0){\makebox(1,1){$2$}}
\put(1,0){\makebox(1,1){$2$}}
\put(2,0){\makebox(2,1){$\cdots$}}

\end{picture}.
\end{center}

If $s_{2}^{\prime}=-$, then the configuration of $T^{\prime}$ is
\setlength{\unitlength}{12pt}

\begin{center}
\begin{picture}(4,1)

\put(0,0){\line(0,1){1}}
\put(1,0){\line(0,1){1}}
\put(2,0){\line(0,1){1}}
\put(0,0){\line(1,0){4}}
\put(0,1){\line(1,0){4}}

\put(0,0){\makebox(1,1){$1$}}
\put(1,0){\makebox(1,1){$2^{\prime}$}}
\put(2,0){\makebox(2,1){$\cdots$}}

\end{picture}.
\end{center}
Note that $s_{1}^{\prime}=+$ so $T_{1,1}^{\prime}=1$.
This contradicts the previous configuration because the configuration of the first row of $\Tilde{e}_{\Bar{1}}^{\pm P}T$ is
\setlength{\unitlength}{12pt}

\begin{center}
\begin{picture}(4,1)

\put(0,0){\line(0,1){1}}
\put(1,0){\line(0,1){1}}
\put(2,0){\line(0,1){1}}
\put(0,0){\line(1,0){4}}
\put(0,1){\line(1,0){4}}

\put(0,0){\makebox(1,1){$1$}}
\put(1,0){\makebox(1,1){$2$}}
\put(2,0){\makebox(2,1){$\cdots$}}

\end{picture}.
\end{center}

In the case when $s_{2}=-$, the configuration of the first row of $T$ is
\setlength{\unitlength}{12pt}

\begin{center}
\begin{picture}(4,1)

\put(0,0){\line(0,1){1}}
\put(1,0){\line(0,1){1}}
\put(2,0){\line(0,1){1}}
\put(0,0){\line(1,0){4}}
\put(0,1){\line(1,0){4}}

\put(0,0){\makebox(1,1){$2^{\prime}$}}
\put(1,0){\makebox(1,1){$2$}}
\put(2,0){\makebox(2,1){$\cdots$}}

\end{picture}.
\end{center}
If $s_{2}^{\prime}=-$, then the configuration of the first row of $T^{\prime}$ is
\setlength{\unitlength}{12pt}

\begin{center}
\begin{picture}(4,1)

\put(0,0){\line(0,1){1}}
\put(1,0){\line(0,1){1}}
\put(2,0){\line(0,1){1}}
\put(0,0){\line(1,0){4}}
\put(0,1){\line(1,0){4}}

\put(0,0){\makebox(1,1){$1^{\prime}$}}
\put(1,0){\makebox(1,1){$2^{\prime}$}}
\put(2,0){\makebox(2,1){$\cdots$}}

\end{picture}.
\end{center}
Note that $s_{1}^{\prime}=-$ so $T_{1,1}^{\prime}=1^{\prime}$.
This contradicts the previous configuration because the configuration of the first row of $\Tilde{e}_{\Bar{1}}^{\pm P}T$ is
\setlength{\unitlength}{12pt}

\begin{center}
\begin{picture}(4,1)

\put(0,0){\line(0,1){1}}
\put(1,0){\line(0,1){1}}
\put(2,0){\line(0,1){1}}
\put(0,0){\line(1,0){4}}
\put(0,1){\line(1,0){4}}

\put(0,0){\makebox(1,1){$1^{\prime}$}}
\put(1,0){\makebox(1,1){$2$}}
\put(2,0){\makebox(2,1){$\cdots$}}

\end{picture}.
\end{center}

In both cases, we have that $s_{2}^{\prime}=+$.
\end{itemize}

\item[(ii):]
$\boldsymbol{a}_{1}\neq\emptyset$.
In this case, $s_{2}=-$.
It is clear that $s_{1}^{\prime}=s_{1}$ and $s_{2}^{\prime}=0$ when $\left\vert \boldsymbol{a}_{2}\right\vert =1$.
When $\left\vert \boldsymbol{a}_{2}\right\vert \geq2$, it is also clear that $s_{1}^{\prime}=s_{1}$.
We claim that $s_{2}^{\prime}=+$.
Since $s_{2}=-$ and $k^{(2)}=1$, the configuration of the first row of $T$ is
\setlength{\unitlength}{12pt}

\begin{center}
\begin{picture}(9,2)

\put(0,0){\line(0,1){1}}
\put(1,0){\line(0,1){1}}
\put(2,0){\line(0,1){1}}
\put(4,0){\line(0,1){1}}
\put(5,0){\line(0,1){1}}
\put(6,0){\line(0,1){1}}
\put(7,0){\line(0,1){1}}
\put(0,0){\line(1,0){9}}
\put(0,1){\line(1,0){9}}

\put(0,0){\makebox(1,1){$x$}}
\put(1,0){\makebox(1,1){$1$}}
\put(2,0){\makebox(2,1){$\cdots$}}
\put(4,0){\makebox(1,1){$1$}}
\put(5,0){\makebox(1,1){$2^{\prime}$}}
\put(6,0){\makebox(1,1){$2$}}
\put(7,0){\makebox(2,1){$\cdots$}}

\put(5,1){\makebox(1,1){\small $r+1$}}
\end{picture},
\end{center}
where $r\geq 1$ and $x$ is either $1^{\prime}$ or $1$.
If $s_{2}^{\prime}=-$, then the configuration of the first row of $T^{\prime}$ is
\setlength{\unitlength}{12pt}

\begin{center}
\begin{picture}(8,2)

\put(0,0){\line(0,1){1}}
\put(1,0){\line(0,1){1}}
\put(2,0){\line(0,1){1}}
\put(4,0){\line(0,1){1}}
\put(5,0){\line(0,1){1}}
\put(6,0){\line(0,1){1}}
\put(0,0){\line(1,0){8}}
\put(0,1){\line(1,0){8}}

\put(0,0){\makebox(1,1){$x$}}
\put(1,0){\makebox(1,1){$1$}}
\put(2,0){\makebox(2,1){$\cdots$}}
\put(4,0){\makebox(1,1){$1$}}
\put(5,0){\makebox(1,1){$2^{\prime}$}}
\put(6,0){\makebox(2,1){$\cdots$}}

\put(4,1){\makebox(1,1){\small $r+1$}}
\end{picture}.
\end{center}
This contradicts the previous configuration because the configuration of the first row of $\Tilde{e}_{\Bar{1}}^{\pm P}T$ is
\setlength{\unitlength}{12pt}

\begin{center}
\begin{picture}(8,2)

\put(0,0){\line(0,1){1}}
\put(1,0){\line(0,1){1}}
\put(2,0){\line(0,1){1}}
\put(4,0){\line(0,1){1}}
\put(5,0){\line(0,1){1}}
\put(6,0){\line(0,1){1}}
\put(0,0){\line(1,0){8}}
\put(0,1){\line(1,0){8}}

\put(0,0){\makebox(1,1){$x$}}
\put(1,0){\makebox(1,1){$1$}}
\put(2,0){\makebox(2,1){$\cdots$}}
\put(4,0){\makebox(1,1){$1$}}
\put(5,0){\makebox(1,1){$2$}}
\put(6,0){\makebox(2,1){$\cdots$}}

\put(4,1){\makebox(1,1){\small $r+1$}}
\end{picture}.
\end{center}
Consequently, we have that $s_{2}^{\prime}=+$.
\end{itemize}

The verification of the rule of Theorem~\ref{thm:eF} is now complete. 
\end{proof}

\subsection*{Acknowledgements}
The author would like to express his gratitude to Professor Susumu Ariki for suggesting the problem and helpful discussions.
He is also grateful to Professor Sami Assaf for kindly pointing out reference~\cite{AO1,AO2} and an anonymous referee for pointing out numerous errors, which was of great help to improve this paper.

\end{document}